\documentclass[12pt]{amsart}

\usepackage{graphicx}
\usepackage{amsfonts}
\usepackage{mathscinet}
\usepackage{epsf}
\usepackage{amssymb}
\usepackage{amsmath}
\usepackage{amscd}
\usepackage{hyperref}
\usepackage{color}
\usepackage{bbm}
\usepackage{tikz}
\usepackage{tikz-cd}
\usepackage{dsfont} 
\usepackage{calligra}
\usepackage{mathrsfs}

\newtheorem{theorem}{Theorem}
\numberwithin{theorem}{section}
\newtheorem{proposition}[theorem]{Proposition}
\newtheorem{lemma}[theorem]{Lemma}
\newtheorem{corollary}[theorem]{Corollary}

\theoremstyle{remark}

\newtheorem{example}[theorem]{Example}
\newtheorem{remark}{Remark}[section] 

\newtheorem{question}{Question}

\newcommand{\kerr}{\text{Ker}}

\newcommand{\imm}{\text{Im}}

\newcommand{\sig}{{\tt SIG }}
\newcommand{\sigg}{{\tt SIG}}
\newcommand{\sigs}{{\tt SIGs }}

\input xy
\xyoption{all}

\makeatletter
\@namedef{subjclassname@2020}{\textup{2020} Mathematics Subject Classification}
\makeatother

\begin{document}

\title[On $\ell_1$ embeddings of finite metric spaces, and SIG graphs]{On $\ell_1$ embeddings of finite metric spaces, and sphere-of-influence graphs}
\begin{abstract}
We introduce the {\em pair-cut cone $PCUT_n$} of metrics on sets with $n\ge 3$ elements, that correspond to linear combinations with non-negative coefficients of the cut-metrics resulting from cuts that are pairs. Given a metric, we fully characterize membership in the pair-cut cone in terms of quantities computed from the metric directly. We also prove a new result by which a metric $d$ that satisfies a system of inequalities, lies in the (full) cut cone of metrics, making it $\ell_1$-embeddable into Euclidean space. 

We give applications of our results to the $\ell_1$-embeddability of simple graphs into Euclidean space as {\em sphere-of-influence graphs}. We exhibit an example of a simple graph that admits no such $\ell_1$-metric in the pair-cut cone.   
\end{abstract}
\author{Stanislav Jabuka}
\email{jabuka@unr.edu}
\author{Ehsan Mirbagheri} 
\email{mirbagherisme@gmail.com}
\address{Department of Mathematics and Statistics, University of Nevada, Reno NV 89557.}

\maketitle

\section{Introduction and Main Results} 
This work concerns itself with two related problems. The first regards the question of when a metric space $(V_n,d)$, with $V_n = \{1,\dots, n\}$, can isometrically be embedded into some Euclidean space  equipped with the $\ell_1$-metric. The second considers simple graphs $G $ and asks if they can be realized as sphere-of-influence graphs with respect to the $\ell_1$-metric in Euclidean space. While it is known that evey simple graph is the sphere-of-influence graph in some metric space, the realization question in the $\ell_1$-metric in Euclidean spaces, remains open. 
\vskip3mm  
\subsection{$\ell_1$-embeddability and the pair-cut cone} \label{Subsection on the ell1-embeddability and the pair-cut cone}
Let $d$ be a semi-metric on the finite set $V_n = \{1,\dots, n\}$, with $n\ge 3$. The question of whether $(V_n,d)$ admits an $\ell_1$-embedding into some Euclidean space $\mathbb R^m$ has a long history with contributions by many authors, see for example \cite{Blumenthal, Caley, DezaLaurentBook, Menger, Schoenberg}. $\ell_1$-embeddability is known to be equivalent to $d$ belonging to the cut cone $CUT_n$, the positive cone spanned by {\em cut-metrics} $\delta_C:V_n\times V_n\to\{0,1\}$, which for a subset (or {\em cut}) $C\subset V_n$ are defined as 
\begin{align} \label{Definition of cut semi-metric from introduction}
	\delta_C(i,j) = \left\{ 
	\begin{array}{cl} 
		0 &\quad ; \quad |\{i,j\}\cap C| = 0\text{ or } 2,  \cr
		& \cr 
		1 &\quad ; \quad |\{i,j\}\cap C| = 1.  
	\end{array} 
	\right. 
\end{align}
We simultaneously view $\delta_C$ as a semi-metric on $V_n$ and as a vector in $\mathbb R^{n\choose 2}$. The latter space has dimension equal to the number of pairs in $V_n$, and with the lexicographic ordering of pairs, the $\{i,j\}$-th coordinate of the vector $\delta_C$ equals $\delta_C(\{i,j\})$. 
 
The problem of deciding whether $d$ lies in $CUT_n$ is known to be NP-complete \cite{AvisDeza}, and thus challenging. Many inequalities are known to obstruct a metric $d$ from belonging to $CUT_n$, including for example the hypermetric inequalities and the null-type inequalities (cf. Section 6.1 in  \cite{DezaLaurentBook}). Even more specialized inequalities can be found in  \cite{DezaLaurant1, DezaLaurant2}. Despite this wealth of obstructions, it remains difficult to decide if a metric $d$ lies in the cut cone.  
\vskip3mm
We propose here the study of a simpler problem, that of determining whether a metric $d$ belongs to the {\em pair-cut cone $PCUT_n$}, defined as the positive cone in $\mathbb R^{n\choose 2}$ spanned by all ${n \choose 2}$ cut-metrics $\delta_C$ with $|C|=2$, referred to as {\em pair-cuts}. 

Studying cut-metrics $\delta_C$ associated to special cuts $C\subset V_n$ rather than all possible cuts, is not a new idea. For example, in \cite{DezaLaurant3}, cuts of only even or odd type are studied. Similarly, in \cite{DezaLaurantGrotschel}, various types of cuts are considered, including equicuts, multicuts, equimulticuts, and balanced multicuts. Our approach of studying only pair-cuts is new, and we offer a complete characterization of membership in  $PCUT_n$. 
\begin{theorem} \label{Main theorem stated in the introduction}
For a metric $d$ on $V_n$ with $n\ge 5$, define the trace Tr$(d)$ of the metric $d$, and the star-trace $s_i$ of a vertex $i\in V_n$ as:  
\begin{itemize}
	\item[(i)] Tr$(d) = \sum _{1\le i<j\le n} d(i,j)$, 
	\vskip2mm
	\item[(ii)] $s_i = \sum_{j\ne i} d(i,j)$.  
\end{itemize}
Then $d$ lies in the pair-cut cone $PCUT_n$ if and only if 
\begin{equation} \label{Main inequality for the coordinates of wsq to be nonnegative from the introduction}
	s_i + s_j\ge (n-4)\cdot d(i,j) + \frac{2}{n-2}\text{Tr}(d), \qquad \text{ for all pairs } \{i, j\} \subset V_n. 
\end{equation}
\end{theorem} 
To our knowledge, such a complete characterization of metrics belonging to the positive cone spanned by a specialized collection of cut-metrics is the first of its kind. The following corollary offers an easier to verify condition, one that is only necessary but generally not sufficient for $d$ to lie in $PCUT_n$. 
\begin{corollary} \label{Corollary with main ineqaulity for PCUTn membership stated in the introduction}
	If a metric $d$ on $V_n$ with $n\ge 5$ lies in the pair-cut cone $PCUT_n$, then 
	\begin{equation} \label{Secondary inequality on star-traces and traces of metrics in the pair-cut-cone}
		(n-2)s_i \ge  \text{Tr}(d), \quad \text{ for all } i\in V_n. 
	\end{equation}
\end{corollary}
Applications of Theorem \ref{Main theorem stated in the introduction} and Corollary \ref{Corollary with main ineqaulity for PCUTn membership stated in the introduction}, are given in Section \ref{Section on the square cut matrix}, specifically in Section \ref{Section with examples and applications of embeddability in PCUT_n}. We are able to prove Theorem \ref{Main theorem stated in the introduction} by a careful spectral analysis of a ${n\choose 2}\times {n \choose 2}$ matrix $S_{sq}$ that we call the {\em square cut-matrix}, and which is the underlying matrix of the linear system determining whether $d$ lies in $PCUT_n$. More than just solving this linear system, we determine exactly when the unique solution has all non-negative coefficients. 
\subsection{$\ell_1$-embeddability and the cut cone}
Membership in the full cut cone $CUT_n$ is likewise determined by searching for solutions 
to a linear system  whose associated matrix $S = S_n$, which we call the {\em full cut-matrix}, has size $\left( 2^{n}-2\right)\times {n\choose 2}$. The number $2^n-2$ of rows of this matrix corresponds to the number of ``non-trivial'' cuts of $V_n $, that is all subsets of $V_n$ other than the empty set and $V_n$ itself. We were not able to prove an analogue of Theorem \ref{Main theorem stated in the introduction} in this general case. Nevertheless, we are able to offer this result. 
\vskip3mm
\begin{theorem} \label{Second main theorem for the full cut matrix as stated in the introduction}
Let $d$ be a metric on $V_n$ with $n\ge 5$. For a non-trivial cut $C\subset V_n$, define its {\em star trace $s_C$} as $s_C =  \sum_{i\in C,\, j\notin C}d(i,j)$. If the inequality  
\begin{equation} \label{Sufficient condition for embedding in full cut cone as stated in the introduction}
	s_C \ge \frac{|C|(n-|C|)}{{n\choose 2} + 1} \cdot \text{Tr}(d)	
\end{equation} 
holds for every non-trivial cut $C\subset V_n$, then $d\in CUT_n$. 
\end{theorem}
Observe that if $C=\{i\}$, then $s_C=s_i$, the star trace of the $i$-th vertex $i\in V_n$, defined in Theorem \ref{Main theorem stated in the introduction}. This theorem is opposite in spirit to the many inequalities mentioned above (e.g. the hypermetric inequalities), which, if violated, obstruct a  metric $d$ from belonging to the cut cone $CUT_n$. In contrast, if a metric $d$ satisfies the inequalities \eqref{Sufficient condition for embedding in full cut cone as stated in the introduction}, that is sufficient for $d$ to lie in $CUT_n$. 
\subsection{Sphere-of-Influence Graphs}  \label{Section in introduction on SIGs}
The second theme we explore is that of {\em sphere-of-influence graphs}, or \sigs for short. Given a simple graph $G=(V_n,E)$ with $E$ the adjacency matrix of $G$ (defined as the $n\times n$ matrix with $E_{i,j}=1$ if the vertices $i$ and $j$ share an edge, and $E_{i,j}=0$ otherwise). A metric $d$ on $V_n$ is said to be a {\em sphere-of-influence} metric for $G$, or a \sigg-metric for $G$,  if the following constraints are satisfied: Firstly, define the {\em radius of influence $r_i$} of the vertex $i\in V_n$ as 
$$r_i = \min_{j\ne i} d(i,j).$$
Then $d$ is a \sigg-metric for $G$ if for any pair $\{i,j\}\subset V_n$, we obtain 
\begin{equation} \label{The SIG inequalities from the introduction}
d(i,j) \left\{ 
\begin{array}{ll}
< r_i+r_j, & \quad ; \quad   E_{i,j} = 1,\cr
\ge r_i+r_j, & \quad ; \quad  E_{i,j} = 0.
\end{array}
\right.  
\end{equation} 
We shall refer to \eqref{The SIG inequalities from the introduction} as the {\em \sigg-inequalities}. 
It is not hard to show that \sigg-metrics exist on every simple graph $G$.
\begin{example} \label{Example of SIG metrics from the introduction}
Given a simple graph $G= (V_n,E)$, let $d_0$ and $d_1$ be the {\em truncated metric} and the {\em graph metric} on $V_n$ respectively, defined as:
$$
d_0(i,j) = \left\{
\begin{array}{cl}
	0 & \quad ; \quad i=j, \cr 
	1 & \quad ; \quad i\ne j \text{ and } E_{i,j} = 1, \cr
	2 & \quad ; \quad i\ne j \text{ and } E_{i,j} = 0.
\end{array}
\right.
$$	
The distance $d_1(i,j)$ is the length of the shortest edge-path connecting the vertices $i$ and $j$, with each edge having length 1. It is easy to check that both metrics are \sigg-metrics for $G$. For some graphs, such as complete graphs, the two metrics coincide, but in general they are different. 

We apply Theorem \ref{Main theorem stated in the introduction} to these two metrics on various families of graphs in Section \ref{Section with examples and applications of embeddability in PCUT_n}. These two metrics were considered in \cite{DezaLaurentBook} in the context of $\ell_1$-embeddings, see ``Part III: Isometric Embeddings of Graphs''. 
\end{example} 
\sigg-graphs were introduced by Toussaint \cite{Toussaint1, Toussaint2, Toussaint3} as a type of {\em proximity graphs} in the Euclidean plane, with an eye toward applications in pattern recognition and computer vision. 
As the subject evolved, \sigg-graphs were considered in arbitrary metric spaces, with special attention heeded to Euclidean spaces $\mathbb R^m$, equipped with the $\ell_p$-metric $d_p$. Recall that for a choice of $p\in [1,\infty]$, $d_p:\mathbb R^m\times \mathbb R^m\to [0,\infty)$ is defined  as 
$$
d_p(x,y) = \left\{  
\begin{array}{ll}
	\left(|x_1-y_1|^p+ \dots+ |x_m-y_m|^p   \right)^{1/p} & \quad ; \quad p\in [1,\infty), \cr 	
	\max\{|x_1-y_1|, \dots, |x_m-y_m|\} & \quad ; \quad p=\infty. 
\end{array} 
\right. 
$$
The following is a natural question in the study of sphere-of-influence graphs. 
\begin{question} \label{Question about L1-embeddability of simple graphs, as stated in the introduction}
Fix a choice of $p\in [1,\infty]$. Does every connected, simple graph $G=(V_n,E)$ admit a \sigg-metric $d$ such that $(V_n,d)$ isometrically embeds into $(\mathbb R^m, d_p)$ for some $m\ge 1$ (with $m$ being allowed to vary with $G$)? 
\end{question}
The current state of the understanding of this problem is this: 
\begin{itemize}
\item[1.] For $p=\infty$, the answer is 'yes'  by work of Michael and Quint \cite{MichaelQuint1}. They provide a beautifully simple embedding of $G$ into $(\mathbb R^{n-1}, d_\infty)$, by mapping the vertex $i\in V_n$ to the $i$-th row-vector of the $n\times (n-1)$ matrix obtained from $E+2I_n$ by removing its $n$-th column.
\item[2.] For $p\in (1,\infty)$, the answer is 'no'. This stems from the strict convexity of the $\ell_p$-norm $||\cdot||_p$ on $\mathbb R^m$. While all $\ell_p$-norms on Euclidean space are convex, they are only stricly convex for $p\in (1,\infty)$ (not considering the trivial case of $m=1$).  
\item[3.] For $p=1$, the question remains open. In this case the question is equivalent (cf. Theorem \ref{Theorem about the equivalence of ell1-embeddability and belonging to the cut cone}) to the question if every simple graph $G=(V_n,E)$ admits a \sigg-metric $d$ that belongs to the cut cone $CUT_n$. 
\end{itemize}
In future work we shall demonstrate that the $p=1$ version of Question \ref{Question about L1-embeddability of simple graphs, as stated in the introduction} is NP-complete, and hence challenging. Given this, and given the introduction of the pair-cut cone in Section \ref{Subsection on the ell1-embeddability and the pair-cut cone}, it is natural to ask this modified question. 
\begin{question} \label{Question about L1-embeddability of simple graphs, as stated in the introduction, into the pair cut cone}
Let $G=(V_n,E)$ be a connected simple graph. Does $G$ admit a \sigg-metric $d$ that lies in the pair cut cone $PCUT_n$? 
\end{question}
We are able to fully answer this simpler question. 
\begin{theorem} \label{Theorem stated in the introduction with an example of a graph none of whose SIG metrics lie in the pair-cut cone}
For every $n\ge 5$ there exists a connected simple graph $G$ with $n$ vertices, such that no \sigg-metric $d$ of $G$ lies in the pair-cut cone $PCUT_n$. One such example is given by the star graph $S(n)$ from Figure \ref{Star Graph from the introduction}.   
\begin{figure}[h]
\includegraphics[width=6cm]{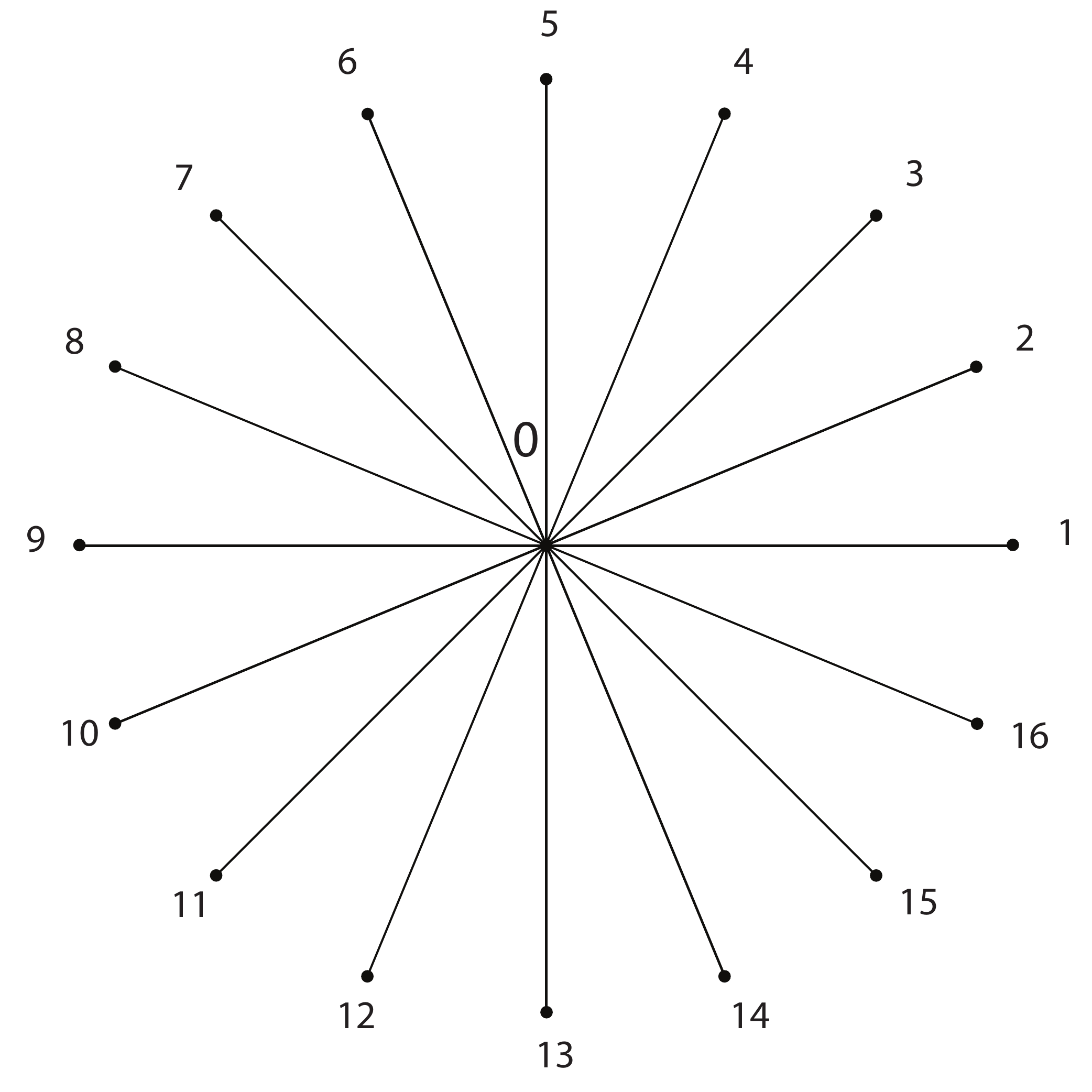}
\caption{The star graph $S(n)$ shown here with $n=16$. The central vertex is labeled $0$ while the peripheral vertices are enumerated by the elements of $V_n=\{1,\dots, n\}$.} \label{Images/StarGraph}
\end{figure} \label{Star Graph from the introduction}
\end{theorem}
The challenge in proving such a theorem is that a graph $G$ typically admits infinitely many \sigg-metrics, and one must exclude all of them from lying in $PCUT_n$. Our proof plays the inequalities \eqref{Main inequality for the coordinates of wsq to be nonnegative from the introduction} from Theorem \ref{Main theorem stated in the introduction}, and the \sigg-inequalities \eqref{The SIG inequalities from the introduction} against each other, and shows them to be incompatible. 
\subsection{Organization} The remainder of this article is organized as this: In Section \ref{Section with background material} we provide necessary background. Specifically, Section \ref{Subsection on finite metric spaces} provides background on metrics on finite sets, and re-casts the $\ell_1$-embeddability problem in terms of solutions to a linear system. Section \ref{Subsection on sphere of influence graphs} gives relevant facts about sphere-of-influence graphs, the second main theme of this work. Section \ref{Section on the square cut matrix} contains the spectral analysis of the square cut-matrix allueded to above, and proves Theorem \ref{Main theorem stated in the introduction} and Corollary \ref{Corollary with main ineqaulity for PCUTn membership stated in the introduction}. Section \ref{Section with applications of main theorem to simple graphs} gives applications of Theorem \ref{Main theorem stated in the introduction} to Question \ref{Question about L1-embeddability of simple graphs, as stated in the introduction, into the pair cut cone}, and gives a proof of Theorem \ref{Theorem stated in the introduction with an example of a graph none of whose SIG metrics lie in the pair-cut cone}. In Section \ref{Section on the full cut matrix} we concern ourselves with the full cut-matrix and we construct a right inverse for it. This right inverse is the main ingredient for proving Theorem \ref{Second main theorem for the full cut matrix as stated in the introduction}. Lastly, in Section \ref{Section on finding an explicit basis for the full cut-matrix}, we determine an explicit basis for the kernel of the full cut-matrix. We explain how this result can be used to modify the results of Theorem \ref{Second main theorem for the full cut matrix as stated in the introduction} to obtain similar membership theorems for metrics in the full cut cone.

\section{Background} \label{Section with background material}
This section gives background on finite metric spaces and sphere-of-influence graphs. Some of this background was already elucidated in the introduction and we do not repeat it here, but trust that the reader recalls that exposition. 
\subsection{Finite Metric spaces} \label{Subsection on finite metric spaces}
For an integer $n\ge 1$, let $V_n = \{1,\dots, n\}$. By a {\em finite metric space} we shall mean a pair $(V_n,d)$ with $d:V_n\times V_n\to [0,\infty)$ a function subject to three conditions:
\vskip2mm
\begin{tabular}{rll}
(a) & Definiteness: & For $i,j \in V_n$, $d(i,j) =0$  if and only if $i=j$. \cr 
(b) & Symmetry: & For $i,j \in V_n$, $d(i,j) = d(j,i)$. \cr 
(c) & Triangle Inequalities: & \begin{minipage}[t]{95mm}For $i,j,k\in V_n$, the triangle inequality $d(i,j) \le d(i,k) + d(k,j)$ holds.\end{minipage} 
\end{tabular}  
\vskip2mm
\noindent If instead of (a), the weaker condition 
\begin{itemize}
	\item[(a')] For every $i\in V_n$, $d(i,i)=0$,  
\end{itemize}
holds, then we call $d$ a {\em semi-metric} on $V_n$. By abuse of terminology we shall still call the pair $(V_n,d)$ a metric space. 

A metric space {\em $(V_n,d)$ is $\ell_1$-embeddable} if there exists an isometric embedding $\varphi :(V_n,d) \to (\mathbb R^m,d_1)$ for some $m\ge 1$. Specifically, if we write $\varphi(i) = X_i \in \mathbb R^m$, then such an embedding requires that %
$$d_1(X_i, X_j) = ||X_i - X_j||_1=  d(i,j), \qquad \text{ for all } i,j\in V_n.$$
\subsection{The Cut Cone $CUT_n$}
A {\em cut} of $V_n$ is a decomposition of $V_n$ as the disjoint union $C\sqcup (V_n-C)$ for some subset $C\subset V_n$. We intuitively think of $C$ as giving a notion of distance among the points of $V_n$ in the sense that points $i,j \in V_n$ that either both lie in $C$ or both lie in $V_n-C$ are ``close'' to each other (say distance 0 apart), while if $i\in C$ and $j\in V_n-C$, or vice versa, then we think of $i$ and $j$ as some positive distance apart, say distance 1. When decomposing $V_n$ as the disjoint union of $C$ and $V_n-C$, we shall refer to $C$ itself as the cut. Note that the cuts $C$ and $V_n-C$ lead to the same decomposition of $V_n$.  

This interpretation of distance in terms of cuts leads to the definition of the {\em cut-metric $\delta_C$}\footnote{The function $\delta_C$ is only a semi-metric, but we find it easier for the exposition to call them cut-metrics.} on $V_n$ as in \eqref{Definition of cut semi-metric from introduction}, clearly with $\delta_C = \delta_{V_n-C}$. The two {\em trivial cuts}, namely $C=\emptyset$ and $C=V_n$, lead to $\delta_C\equiv 0$ and hence we shall mostly only consider the $2^{n}-2$ non-trivial cuts of $V_n$. 

Recall that we regard the semi-metrics $\delta_C$ as vectors in $\mathbb R^{n\choose 2}$, as explained in the introduction in the paragraph following equation \eqref{Definition of cut semi-metric from introduction}. All of their coordinates are either 0 or 1, so that they are in fact vertices of the hypercube $[0,1]^{n\choose 2}$. We define the  {\em cut polytope $CUT_n^\square$} and the {\em cut cone $CUT_n$} as 
\begin{align}
\text{CUT}_n & = \left\{ \sum_{C\subset V_n} w_C \, \delta_C \,\big|\, w_C\ge 0\right\} \subset \mathbb R^{n\choose 2}, \cr  
& \cr 
\text{CUT}_n^\square & = \left\{ \sum_{C\subset V_n} w_C \, \delta_C \,\big|\, w_C\ge 0, \,\, \sum _{C\subset V_n} w_C = 1\right\} \subset CUT_n \subset \mathbb R^{n\choose 2}.  	
\end{align}  
The cut polytope $CUT^\square_n$ is the convex hull Conv$(X)$ in $\mathbb R^{n\choose 2}$ generated by the finite set $X$ of vertices of the hypercube $[0,1]^{n\choose 2}$ that correspond to cut-metrics $\delta_C$. The cut cone $CUT_n$ is the positive cone of $CUT^\square _n$. Our use of the word {\em polytope} is as the convex hull of a finite subset of Euclidean space, and thus applies in this setting. The cut polytope $CUT_n^\square$ is still not well understood. For example, its {\em facets}, the codimension 1 faces, are only fully known for $n\le 7$ (for $n\le 6$ see \cite{DezaLaurant1}, for $n=7$ see \cite{Grishukhin}), but for larger $n$ they remain elusive. 

\begin{example}
When $n=3$, $CUT^\square_3$ is generated in $[0,1]^3\subset \mathbb R^3$ by the 4 vertices 
$$
\begin{array}{rlrl}
\delta_{\{\}} & \hspace{-3mm}=\delta_{\{1,2,3\}}= (0,0,0), & \qquad  \qquad  \delta_{\{1\}}  & \hspace{-3mm}= \delta_{\{2,3\}}= (1, 1, 0), \cr
\delta_{\{2\}} & \hspace{-3mm}= \delta_{\{1,3\}}   = (1, 0, 1),&   \delta_{\{3\}} &\hspace{-3mm}= \delta_{\{1,2\}} = (0, 1, 1),
\end{array} 
$$
showing $CUT^\square_3$ to be the tetrahedron in Figure \ref{CUT3 polytop}.
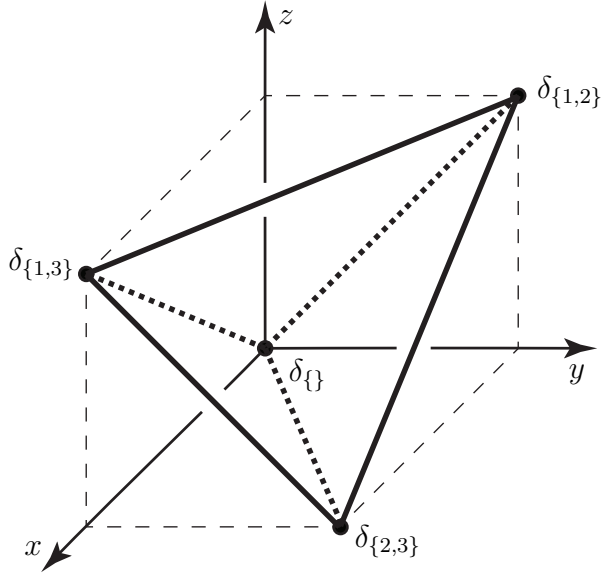
\begin{figure}[h]
\includegraphics[width=8cm]{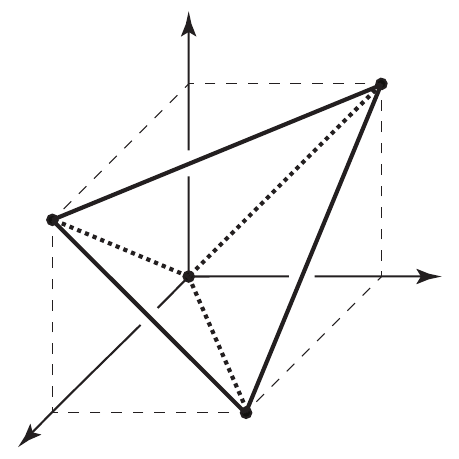}
\put(-100,15){$\delta_{\{2,3\}}$}
\put(-231,120){$\delta_{\{1,3\}}$}
\put(-31,185){$\delta_{\{1,2\}}$}
\put(-125,78){$\delta_{\{\}}$}
\put(-225,12){$x$}
\put(-19,79){$y$}
\put(-129,213){$z$}
\caption{A visualization of $CUT_3^\square$ and $PCUT^\square_3$. The former is the tetrahedron spanned by $\delta_{\{\}}$, $\delta_{\{1,2\}}$, $\delta_{\{1,3\}}$ and $\delta_{\{2,3\}}$, while the latter is the triangle spanned by $\delta_{\{1,2\}}$, $\delta_{\{1,3\}}$ and $\delta_{\{2,3\}}$.} \label{CUT3 polytop}
\end{figure}
\end{example}

The following classical result concerning convex hulls is due to Carath\'eodory. 
\begin{theorem}[Carath\'eodory \cite{Caratheodory}] \label{Caratheodorys Theorem}
Let $X\subset \mathbb R^d$ be a finite set and let $x$ be a point in the convex hull Conv$(X)\subset \mathbb R^d$ of $X$. Then $x$ lies in a $d$-dimensional simplex of $X$. Said differently, there exist $d+1$ linearly independent points $x_0,\dots, x_{d}$ such that 
$$x = \sum_{i=0}^{d} \lambda _i x_i, \qquad \text{ with } \qquad \lambda_i\ge 0 \text{ and } \,\, \sum _{i=0}^{d} \lambda_i =1.$$ 
\end{theorem} 

The connection between $\ell_1$-embeddability of $(V_n,d)$ and the cut cone is the following beautiful result from \cite{Assouad}.
\begin{theorem} \label{Theorem about the equivalence of ell1-embeddability and belonging to the cut cone}
A metric $d$ on $V_n$ is $\ell_1$-embeddable if and only if $d\in CUT_n$. 
\end{theorem}
Indeed, if $d\in CUT_n$ and $d = \sum_{i=1}^m  w_i\, \delta_{C_i}$ for some $w_i\ge 0$ and for some cuts $C_1, \dots, C_m\subset V_n$ (with $m\le {n\choose 2} $), then an $\ell_1$-embedding of $(V_n,d)$ into $\left( \mathbb R^m, d_1\right)$ is given by
$$i\mapsto X_i \quad \text{ with } \quad (X_i)_j=\left\{ 
\begin{array}{cl}
w_j & \quad ; \quad i \in C_j, \cr 
0 & \quad ; \quad i \notin C_j. 	
\end{array} 
\right.  $$
It is easy to verify that $||X_i-X_j||_1 = d(i,j)$ for all $i,j \in V_n$. The converse statement is also not hard to prove, see Proposition 4.2.2 in \cite{DezaLaurentBook}.

This result reduces the question of $\ell_1$-embeddability to that of verifying membership of a given metric $d$ in the cut cone $CUT_n$. This, however, is not an easy problem.  
\begin{theorem} [Avis-Deza \cite{AvisDeza}] The problem of determining if a metric $d$ on $V_n$ belongs to the cut cone $CUT_n$ is NP-complete. 
\end{theorem}
\subsection{Sphere-of-Influence Graphs} \label{Subsection on sphere of influence graphs}
We carry forward here the defintion and symbols introduced in Section \ref{Section in introduction on SIGs} regarding sphere-of-influence graphs or \sigs for short. In that section we started with a simple, connected graph $G = (V_n,E)$ and  showed that there always exist metrics $d$ on $V_n$ whose associated \sig is $G$. 

Conversely, if we begin with a metric $d$ on $V_n$, we can form its sphere-of-influence graph $G$ whose vertices are $V_n$ and with edge matrix $E = [E_{i,j}]$ determined by 
$$ E_{i,j} = \left\{
\begin{array}{cl}
	0 & \quad ; \quad i=j, \cr 
1 & \quad ; \quad i\ne j \text{ and } d(i,j)<r_i+r_j, \cr 
0 & \quad ; \quad  i\ne j \text{ and } d(i,j)\ge r_i+r_j. 
\end{array} 
\right. $$ 
As before, $r_i = \min _{j\ne i} d(i,j)$. By construction, $G$ is a simple graph though it may or may not be connected. 6
\vskip3mm
Before proceeding, we briefly return to Question \ref{Question about L1-embeddability of simple graphs, as stated in the introduction} from the introduction. In the paragraph following it, we showed how every simple graph $G=(V_n,E)$ can be given a \sigg-metric $d$ such that $(V_n,D)$ embeds into $(\mathbb R^{n-1},d_\infty)$. We also stated that such metrics do not exist in general if $p\in (1,\infty)$, and we show here by example why this is. This result is well known to the experts, we only include it for completeness. 

\begin{theorem}
For any $p\in (1,\infty)$ there exist connected, simple graphs $G$ that do not admit an $\ell_p$-\sig embedding into any Euclidean space. 
\end{theorem}

The proof is simple, and a single graph provides a counterexample for all $p\in (1,\infty)$. Recall that if $p\in (1,\infty)$, then $d_p(a,c) = d_p(a,b)+d_p(b,c)$ for points $a,b,c \in \mathbb R^m$ implies that $a, b, c$ are collinear. This fails for $p=1, \infty$ owing to the fact that the unit ball with respect to $d_p$ is strictly convex if and only if $p\in (1,\infty)$, see \cite{Hardy} (Section 2.11) or \cite{Rudin} (Chapter 3). 
\begin{example} 
Consider the star graph $S(3)$ from Figure \ref{GraphY}. Suppose $d$ is a \sigg-metric on $W_4=\{0\}\cup V_3$ with $a_i =  d(i,0)$ for $i=1,2,3$. Since each of the vertices $1,2,3$ have valence 1, it follows that $r_i=a_i$, $i=1,2,3$. 
\begin{figure}[h]
	\includegraphics[width=6cm]{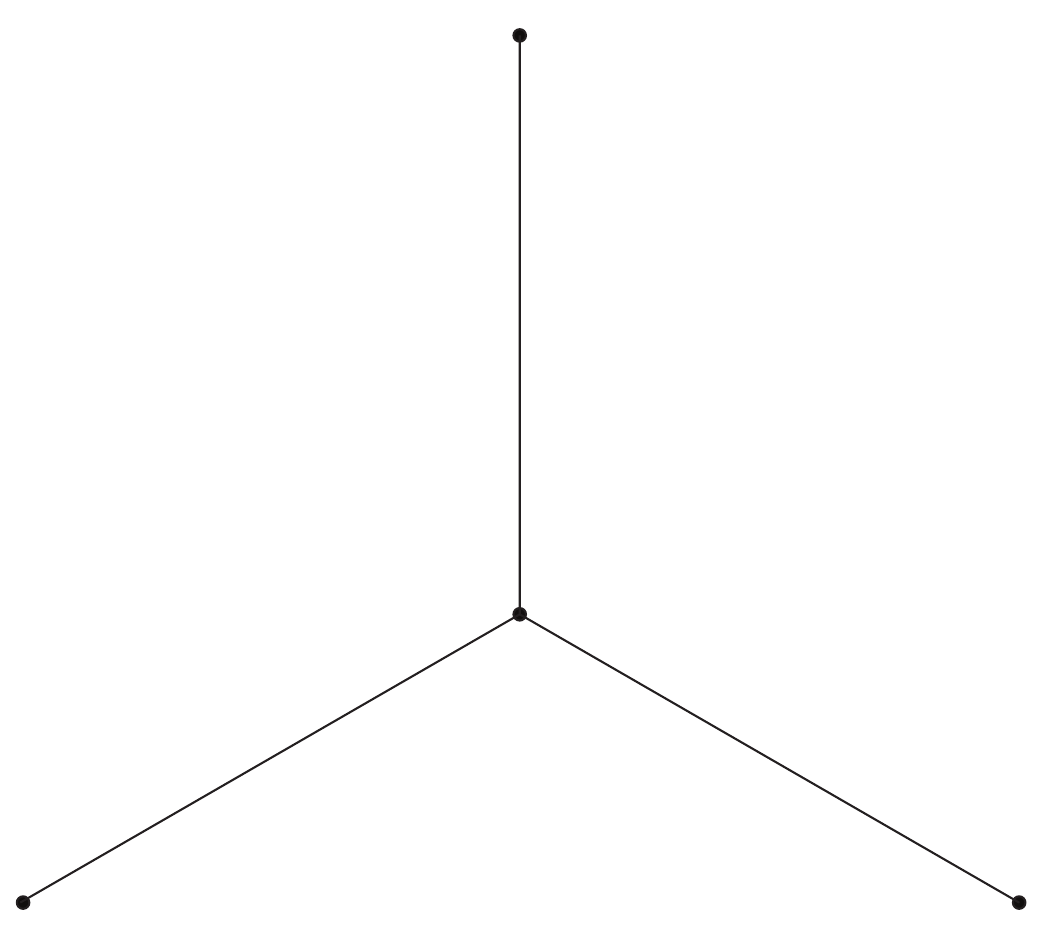}
	\put(-88,155){$3$}
	\put(-178,-3){$1$}
	\put(3,-3){$2$}
	\put(-89,38){$0$}
	\put(-140,33){$a_1$}
	\put(-42,33){$a_2$}
	\put(-82,100){$a_3$}
	\caption{The star graph $S(3)$. } \label{GraphY}
\end{figure}
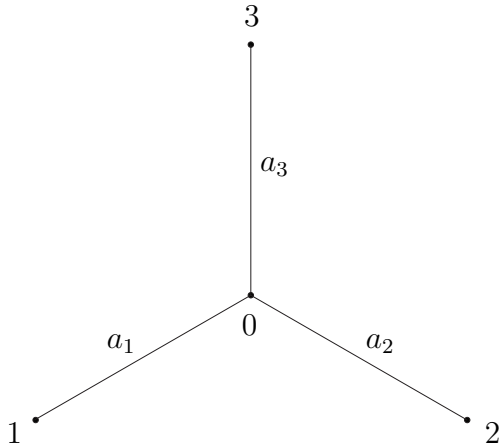
The \sigg-inequalities \eqref{The SIG inequalities from the introduction} applied to pairs of vertices that don't share an edge, along with the triangle inequality, imply that 
$$d(1,2)\ge r_1+r_2 = a_1+a_2 = d(1,0) + d(0,2) \ge d(1,2), $$
showing that $d(1,2) = a_1+a_2 = d(1,0)+d(0,2)$. The same argument shows that $d(1,3) = d(1,0) + d(0,3)$ and $d(2,3) = d(2,0) + d(0,3)$. If $S(3)$ admitted a \sigg-embedding into $(\mathbb R^m, d_p)$ for some $p\in (1,\infty)$, then each of the triples of points $\{1,2,0\}$, $\{1,3,0\}$ and $\{2,3,0\}$ would be collinear, and therefore all 4 points would be collinear, a contradiction.
\end{example} 
As the question of whether all simple graphs admit \sigg-metrics that lie in the cut cone $CUT_n$ remains open, we instead turn to investigating membership of \sigg-metrics in the pair-cut cone $PCUT_n$. 
\section{The Square Cut-Matrix and the Pair-Cut Cone $PCUT_n$} \label{Section on the square cut matrix}
This section introduces the main objects of our study - the {\em pair-cut cone $PCUT_n$} and the {\em square cut-matrix $S_{sq}$}. The culminating results are an explicit spectral decomposition of $S_{sq}$ (Section \ref{Section on the spectral decomposition of the square cut-matrix}) and a complete characterization of a metric $d$ on $V_n$ belonging to $PCUT_n$ (Theorem \ref{Main inequalitis on star-traces and traces of metrics in the pair-cut-cone}). 
\subsection{The pair-cut cone}
Recall that every metric $d\in CUT_n$ is a linear combination 
$$ d = \sum_{i=1}^m w_i\,  \delta_{C_i}, $$
for some cuts $C_1,\dots, C_m$ of $V_n$ and some coefficients $w_i\ge 0$. By Carath\'eodory's Theorem \ref{Caratheodorys Theorem}, we may assume that $m\le {n \choose 2}$ though we cannot know which cuts $C_1,\dots, C_m$ are used for a given metric $d$. 

In this section we make what we see as a natural choice for both $m$ and the cuts $C_1, \dots, C_m$. Namely, we choose to consider:
\begin{itemize}
\item[(i)] $m = {n \choose 2}$.
\item[(ii)] The cuts corresponding to the ${n \choose 2}$ pairs $\{i,j\}\subset V_n$, referred to as {\em pair-cuts}, and which we order lexicographically. We shall label them either as pairs $\{i,j\}$ or as cuts $C_1, C_2, \dots$ where the index refers to their ordinal in the lexicographic ordering. Thus for example 
$$C_1 =\{1,2\}, \,  C_2 = \{1,3\}, \dots , C_{{n\choose 2}} = 
\{n-1,n\}.$$ 
\end{itemize}
With these choices, we shall study metrics $d$ on $V_n$ of the form 
\begin{equation} \label{A pair cut metric}
	d = \sum_{\{i,j\} \subset V_n} w_{\{i,j\}}\cdot \delta_{\{i,j\}}, \qquad w_{\{i,j\}} \ge 0. 
\end{equation}   
Let $PCUT_n^\square$ be the {\em pair-cut polytope}, that is the convex hull of all the pair-cut metrics $\delta_{\{i,j\}}$, and let $PCUT_n$ denote the {\em pair-cut cone}, which is the positive cone on $PCUT^\square _n$. Thus, $d\in PCUT_n$ if and only if \eqref{A pair cut metric} has a solution. With this understood, the main question we address here is this:
\begin{question} \label{Question about a metric belonging to the pair-cut cone}
Under what conditions does a metric $d$ on $V_n$ belong to the pair-cut cone $PCUT_n$? 	
\end{question} 
\subsection{The square cut-matrix}
\noindent We start by rewriting the system \eqref{A pair cut metric} in matrix form. Let 
$$w = [w_{\{1,2\}}, \dots, w_{\{n-1,n\}}]^\tau \in \mathbb R^{n\choose 2}.$$
For a fixed $n\ge 3$, which we drop from the notation, we define the  {\em square cut-matrix $S_{sq}$} to be the ${n \choose 2}\times {n \choose 2}$ square matrix whose columns are the vectors $\delta_{\{1,2\}}, \delta_{\{1,3\}},\dots, \delta_{\{n-1,n\}}$, so that the system \eqref{A pair cut metric} becomes 
\begin{equation} \label{A pair cut metric via matrices}   
S_{sq}\cdot w = d, \quad w\ge 0.
\end{equation} 
Here and elsewhere, writing $w\ge 0$ means that $w$ is a vector all of whose coordinates are non-negative. If we view both the rows and columns of $S_{sq}$ enumerated by the ${n \choose 2}$ pairs in $V_n$ with respect to the lexicographic ordering, then the entry in row $k$ and column $\ell$ of $S_{sq}$ is $\delta_{\{a,b\}}(\{i,j\})$ if $\{a,b\}$ and $\{i,j\}$ are the $k$-th and $\ell$-th pair respectively. In particular, this shows that $S_{sq}$ is a symmetric matrix, and thus has real eigenvalues. 
\begin{example} For $n=4$ the lexicographic ordering of pairs in $V_4=\{1,2,3,4\}$ is $\{1,2\}, \{1,3\}, \{1,4\}, \{2,3\}, \{2,4\}, \{3,4\}$, leading to 
$$S_{sq} = \begin{bmatrix} 0 & 1 & 1 & 1 & 1 & 0  \\ 
1 & 0 & 1 & 1 & 0 & 1 \\
1 & 1 & 0 & 0 & 1 & 1 \\
1 & 1 & 0 & 0 & 1 & 1 \\
1 & 0 & 1 & 1 & 0 & 1 \\
0 & 1 & 1 & 1 & 1 & 0 
\end{bmatrix}.  
$$
\end{example} 
\noindent The eigenvalues of $S_{sq}$, as it turns out, are well known. 
\begin{theorem} \label{Full rank of the cut matrix}
For $n\ge 3$ the eigenvalues of the matrix $S_{sq}$ and their multiplicities, are:  
\begin{itemize}
\item[(i)] $2n-4$ with multiplicity 1, 
\item[(ii)] $n-4$ with multiplicity $n-1$,
\item[(iii)] $-2$ with multiplicity $\frac{1}{2}n(n-3)$. 
\end{itemize}
Accordingly, for $n\ge 5$, $S_{sq}$ is a regular matrix. 
\end{theorem}
The eigenvalues of $S_{sq}$ have been extensively studied in the literature, for example  in \cite{Cvetkovic2} (cf. Theorem 4.17). Recall that spectral graph theory is the study of graphs by means of the eigenvalues of their adjacency matrices, and $S_{sq}$ equals the adjacency matrix $A = A_n$ of the line graph $L(K_n)$ of the complete graph $K_n$ on $n$ vertices. As such it has received much attention in the literature. 
The line graph $L(K_n)$ is the graph whose vertices are in one-to-one correspondence with the edges of $K_n$, and two vertices $e_1, e_2$ of $L(K_n)$ (with $e_1, e_2$ being edges of $K_n$) share an edge in $L(K_n)$ if and only if $e_1$ and $e_2$ share a vertex in $K_n$. Since $K_n$ has ${n\choose 2}$ edges, then $L(K_n)$ has ${n\choose 2}$ vertices, which we shall label with unordered pairs $\{i,j\}\subset V_n$ (corresponding to the unique edge in $K_n$ connecting the vertices $i$ and $j$). 

Let $B=B_n$ be the {\em vertex-edge incidence matrix of $K_n$}. This is a matrix of size $n\times {n\choose 2}$, whose $(k,\ell)$-th entry equals 1 if the $k$-th vertex of $K_n$ is an endpoint of the $\ell$-th edge, otherwise the $(k,\ell)$-th entry equals 0. This matrix too has received considerable attention in spectral graph theory. Specifically, the following identities are well known, but can also easily be verified explicitly:
\begin{align} \label{Equations for the edge-vertex matrix B}
B^\tau B & = 2I_{n\choose 2} + A \cr 
B B^\tau  & = (n-2)I_n + J_n,
\end{align}
Here $I_n$ is the $n\times n$ identity matrix, while $J_n$ is the $n\times n$ matrix all of whose entries equal 1. We shall use these relations to obtain an explicit spectral decomposition of $A=S_{sq}$. 
\subsection{The spectral decomposition of $S_{sq}$}
In this section we explicitly identify the eigenspaces $E_\lambda$ of $A=S_{sq}$, and the orthogonal projectors onto each eigenspace. Recall that the eigenvalues of $A$ and their multiplicities were listed in Theorem \ref{Full rank of the cut matrix}. We assume that $n\ge 5$ so that $A$ is a regular matrix.  
We begin by observing that 
$$v\in \kerr(B) = \kerr(B^\tau B) \text{ if and only if } Av = -2v \quad (\text{by } \eqref{Equations for the edge-vertex matrix B})$$
and thus $\kerr(B) = E_{-2}.$
By Theorem \ref{Full rank of the cut matrix}, $\dim \kerr(B) = \frac{1}{2}n(n-3)$. Since eigenspaces of a symmetric matrix belonging to distinct eigenvalues are orthogonal to each other, we find that  
$$E_{2n-4}\oplus E_{n-4} = E_{-2}^\perp = (\kerr(B))^\perp = \text{Col}(B^\tau),$$
where Col$(B^\tau)$ is the column space of $B^\tau$. By uniqueness of the orthogonal complement, it follows that $\text{Col}(B^\tau) = E_{2n-4}\oplus E_{n-4}$, leading to this orthogonal decomposition: 
\begin{equation} \label{Decomposition into kernel of B and column space of B transpose}
\mathbb R^{n\choose 2} = \kerr(B) \oplus \text{Col}(B^\tau) = E_{-2}\oplus \left(E_{2n-4}\oplus E_{n-4} \right).
\end{equation} 
For an eigenvalue $\lambda $ of $A$, let $P_\lambda :\mathbb R^{n\choose 2} \to \mathbb  R^{n\choose 2}$ be the orthogonal projector onto $E_\lambda$. Recall that in general, an operator $P :\mathbb R^{n\choose 2} \to \mathbb R^{n\choose 2}$ is an orthogonal projector (onto $\imm(P)$) if and only if $P^2 = P = P^\ast$, where $P^\ast$ is the adjoint of $P$ and is represented by the transpose of the matrix that represents $P$. We turn to determining $P_{2n-4}$, $P_{n-4}$ and $P_{-2}$ next. 

\begin{lemma}
The orthogonal projector $P_{\text{Col}}:\mathbb R^{n\choose 2} \to R^{n\choose 2}$ onto Col$(B^\tau)$, is given by 
$$P_{\text{Col}}=B^\tau (BB^\tau)^{-1}B.$$
\end{lemma}
\begin{proof} 
Use \eqref{Decomposition into kernel of B and column space of B transpose} to write any $v\in \mathbb R^{n\choose 2}$ uniquely as $v = v_1 + v_2$ with $v_1\in \kerr(B), v_2\in \text{Col}(B^\tau)$.Then   
\begin{align*}
B^\tau(BB^\tau)^{-1}B & (v_1+v_2)  = B^\tau(BB^\tau)^{-1}B (v_2) \cr 
& = B^\tau(BB^\tau)^{-1}B (\lambda_1B_1^\tau + \dots + \lambda_nB_n^\tau) \qquad (B^\tau_i = i\text{-th column of }B^\tau)\cr 
& = \sum_{i=1}^n \lambda _i B^\tau (BB^\tau)^{-1} BB_i^\tau \cr
& = \sum_{i=1}^n \lambda _i B^\tau (BB^\tau)^{-1} (BB^\tau)_i \quad ((BB^\tau)_i = i\text{-th column of }BB^\tau) \cr 
& = \sum_{i=1}^n \lambda _i B^\tau \cdot [0,\dots, 0,\underbrace{1}_{i-th \text{ coordinate}},0\dots, 0]^\tau  \cr
& = \sum_{i=1}^n \lambda _i B_i^\tau \cr 
& = v_2.
\end{align*}
Additionally
\begin{align*}
(B^\tau(BB^\tau)^{-1}B)^2& = B^\tau(BB^\tau)^{-1}(BB^\tau)(BB^\tau)^{-1}B = B^\tau(BB^\tau)^{-1}, \cr
(B^\tau(BB^\tau)^{-1}B)^\tau & = B^\tau (BB^\tau)^{-\tau} B = B^\tau(BB^\tau)^{-1}B.
\end{align*}
This completes the proof.  
\end{proof}
Knowing the orthogonal projector $P_{\text{Col}}$ onto the column space of $B^\tau$, we can derive the orthogonal projectors onto the eigenspaces of $A$. 
\begin{proposition} \label{The orthogonal projectors of the pair-cut matrix}
The orthogonal projectors $P_\lambda$ onto the eigenspaces $E_\lambda$ of $A$, for $\lambda = n-4, 2n-4, -2$ are given by:
\begin{itemize}
\item[(i)] $P_{-2} = I_{n\choose 2} - P_{\text{Col}} = I_{n\choose 2} - B^\tau(BB^\tau)^{-1}B$.
\vskip3mm
\item[(ii)] $P_{2n-4} = \frac{1}{{n\choose 2}} \, J_{n\choose 2}$.
\vskip3mm
\item[(iii)] $P_{n-4} = P_{\text{Col}} - P_{2n-4} = B^\tau(BB^\tau)^{-1}B - \frac{1}{{n\choose 2}}\,  J_{n\choose 2}$.
\end{itemize}
\end{proposition} 
\begin{proof}
Since $E_{-2} = \kerr(B) = (\text{Col}(B^\tau)^\perp$, then the orthogonal projector $P_{-2}$ onto $E_{-2}$ is given by $P_{-2} = I_{n\choose 2} - P_{\text{Col}}$, as claimed in (i).
\vskip3mm 
The sole generator of $E_{2n-4}$ is easily seen to be given by $\mathbf{1}_{n\choose 2}$ (the column vector of dimension ${n\choose 2}$ with all coordinates equal to 1), and therefore the orthogonal projector $P_{2n-4}$ onto $E_{2n-4}$ is given by 
$$P_{2n-4} = \lambda \mathbf{1}_{n\choose 2}\cdot \mathbf{1}^\tau_{n\choose 2} = \lambda  J_{n\choose 2},$$
for some choice of $\lambda \in \mathbb R$. 
Clearly $P_{2n-4}^\ast = P_{2n-4}$ for any choice of $\lambda$, but $P_{2n-4}^2 = \lambda ^2 J_{n\choose 2}^2 = \lambda^2 {n\choose 2} J_{n\choose 2}$, showing that $\lambda = \frac{1}{{n\choose 2}}$. This leads to $P_{2n-4} = \frac{1}{{n\choose 2}}\, J_{n\choose 2}$, proving (ii). 
\vskip3mm  
Since Col$(B^\tau) = E_{n-4}\oplus E_{2n-4}$, then the orthogonal projector $P_{n-4}$ onto $E_{n-4}$ is gotten by first orthogonally projecting onto Col$(B^\tau)$ using the already computed $P_{\text{Col}}$, and then applying $I_{n\choose 2} - P_{2n-4}$. This leads to
$$P_{n-4} = \left(I_{n\choose 2} - P_{2n-4}\right) P_{\text{Col}} = P_{\text{Col}} - P_{2n-4} = B^\tau(BB^\tau)^{-1}B - \frac{1}{{n\choose 2}}\, J_{n\choose 2}.$$ 
In the above we used the fact that $E_{2n-4}\subset \text{Col}(B^\tau)$ and thus $P_{2n-4}P_{\text{Col}} =P_{2n-4}$. Lastly, note that 
\begin{align*}
P_{n-4}^\ast  & = (P_{\text{Col}} - P_{2n-4})^\ast = P_{\text{Col}}^\ast - P_{2n-4}^\ast = P_{\text{Col}} - P_{2n-4} = P_{n-4}, \cr 
P_{n-4}^2 & = (P_{\text{Col}} - P_{2n-4})^2 = P_{\text{Col}}^2 - P_{\text{Col}} P_{2n-4} - P_{2n-4}P_{\text{Col}} +P_{2n-4}^2 \cr
& = P_{\text{Col}} - P_{\text{Col}} P_{2n-4} - P_{2n-4}P_{\text{Col}} +P_{2n-4} \cr
& = P_{\text{Col}} - P_{2n-4} - P_{2n-4}+P_{2n-4} \cr
& = P_{\text{Col}} - P_{2n-4}\cr
& = P_{n-4}.
\end{align*} 
This proves (iii), thereby completing the proof of the proposition. 
\end{proof}
\subsection{The spectral decomposition of $S_{sq}$ and $S_{sq}^{-1}$} \label{Section on the spectral decomposition of the square cut-matrix}
With the above calculations of the eigenspaces $E_\lambda$ of the square cut-matrix $S_{sq} = A$ and their associated orthogonal projectors $P_\lambda :\mathbb R^{n\choose 2} \to \mathbb R^{n\choose 2}$, we now use the Spectral Theorem to obtain this identity: 
$$S_{sq} = A = -2P_{-2}+ (2n-4) P_{2n-4}+ (n-4)P_{n-4}.$$
Continuing with our assumption of $n\ge 5$, we see that $A$ is invertible and therefore the Spectral Theorem applied to $A^{-1}$ yields: 
\begin{equation} \label{First expression for the invers of the square cut matrix}
A^{-1} = -\frac{1}{2}P_{-2} + \frac{1}{2n-4}P_{2n-4}+\frac{1}{n-4}P_{n-4}.
\end{equation} 

By relying on Proposition \ref{The orthogonal projectors of the pair-cut matrix}, we may simplify \eqref{First expression for the invers of the square cut matrix} as follows:
\begin{align} \label{Second expression for the invers of the square cut matrix}
A^{-1} & = -\frac{1}{2}P_{-2} +\frac{1}{2n-4}P_{2n-4}+\frac{1}{n-4}P_{n-4},	\cr
& = -\frac{1}{2} \left(I_{n\choose 2} - P_{\text{Col}}\right) +\frac{1}{2n-4} P_{2n-4} +\frac{1}{n-4} \left( P_{\text{Col}} - P_{2n-4}\right), \cr
& = -\frac{1}{2} I_{n\choose 2} -\frac{n}{2(n-2)(n-4)} P_{2n-4} + \frac{n-2}{2(n-4)} P_{\text{Col}}.
\end{align}
\subsection{Solving the system $S_{sq}w = d$}
The unique solution $w_{sq}$ of $S_{sq}w_{sq} = Aw_{sq} = d$ (this is the linear system \eqref{A pair cut metric via matrices}) is courtesy of \eqref{Second expression for the invers of the square cut matrix} and Proposition \ref{The orthogonal projectors of the pair-cut matrix}, given by 
\begin{align} \label{First formula for wsq from orthogonal projectors}
w_{sq} & = A^{-1}\,d, \cr 
& = -\frac{1}{2}d - \frac{n}{2(n-2)(n-4)}P_{2n-4}(d)  + \frac{n-2}{2(n-4)} P_{\text{Col}}(d), \cr
& = -\frac{1}{2}d - \frac{n}{2(n-2)(n-4)}\cdot \frac{2}{n(n-1)} J_{n\choose 2} d  + \frac{n-2}{2(n-4)} P_{\text{Col}}(d), \cr
 & = -\frac{1}{2}d - \frac{1}{(n-1)(n-2)(n-4)}\text{Tr}(d) \mathbf{1}_{n\choose 2}  + \frac{n-2}{2(n-4)} P_{\text{Col}}\, d.
\end{align}
In the above we have introduced the notation Tr$(d)$, which we refer to as the {\em trace of the metric $d$}, and which is defined as 
\begin{equation} \label{Definition of the trace of a metric}
\text{Tr}(d) = \sum_{i<j} \, d(i,j).
\end{equation}
The first two terms in the last row of \eqref{First formula for wsq from orthogonal projectors} are completely explicit, while the third term is not. We rectify this next. From the definition of the vertex-edge matrix $B$, we find for each $i\in V_n$: 
$$(Bd)_i  = \sum_{k\ne i} d(i,k).  $$
As the type of sum appearing on the right-hand side above makes a frequent appearance in formulae below, we call it the {\em star-shaped trace $s_i$ of the $i$-th vertex $i\in V_n$}, that is  
\begin{equation} \label{Definition of the star-shaped trace of the i-th vertex}
s_i = \sum_{k\ne i} d(i,k).
\end{equation} 
The star-shaped trace is the sum of the distances from the $i$-th vertex to all other vertices $k\in V_n-\{i\}$. Note that $(Bd)_i = s_i$ and that Tr$(d) =\frac{1}{2} \sum _{i=1}^n s_i$. Next we calculate $(BB^\tau)^{-1}(Bd)_i$ by relying on \eqref{Equations for the edge-vertex matrix B}, from which one finds that 
$$(BB^\tau)^{-1} =  \frac{1}{n-2}I_n - \frac{1}{2(n-1)(n-2)} J_n.$$
Thus
\begin{align*}
(BB^\tau)^{-1}(Bd) & =  \left( \frac{1}{n-2}I_n - \frac{1}{2(n-1)(n-2)} J_n\right)(Bd), \cr 
& = \frac{1}{n-2}(Bd) - \frac{1}{2(n-1)(n-2)} J_n (Bd),  \cr
& = \frac{1}{n-2}(Bd)  - \frac{1}{2(n-1)(n-2)} \left( \sum_{i=1}^n s_i\right)  \mathbf{1}_n,  \cr
& = \frac{1}{n-2}(Bd)  - \frac{1}{(n-1)(n-2)} \text{Tr}(d) \mathbf{1}_n.  \cr
\end{align*}
Lastly, we apply $B^\tau$ to both sides and again rely on \eqref{Equations for the edge-vertex matrix B}:
\begin{align*}
P_{\text{Col}}(d) &= B^\tau (BB^\tau)^{-1}(Bd), \cr
& = \frac{1}{n-2}(B^\tau B)d  - \frac{1}{(n-1)(n-2)} \text{Tr}(d) B^\tau \mathbf{1}_n, \cr 
& = \frac{1}{n-2}(2I_{n\choose 2} + A)d  - \frac{1}{(n-1)(n-2)} \text{Tr}(d) B^\tau \mathbf{1}_n, \cr 
& = \frac{2}{n-2}d + \frac{1}{n-2}(Ad)  - \frac{1}{(n-1)(n-2)} \text{Tr}(d) (B^\tau \mathbf{1}_n). 
\end{align*}
The coordinate of the right-hand side corresponding to the pair $\{i,j\}$ is not hard to evaluate. We use the notation $A_{\{i,j\}}$ to mean the row of $A$ corresponding to the pair $\{i,j\}$: 
\begin{align*}
(P_{\text{Col}}(d))_{\{i,j\}} & = \frac{2}{n-2}d(i,j) + \frac{1}{n-2}(Ad)_{\{i,j\}} - \frac{\text{Tr}(d)}{(n-1)(n-2)} (B^\tau \mathbf{1}_n)_{\{i,j\}}, \cr
& = \frac{2}{n-2}d(i,j) + \frac{1}{n-2}A_{\{i,j\}}\cdot d - \frac{\text{Tr}(d)}{(n-1)(n-2)} (B^\tau)_{\{i,j\}} \mathbf{1}_n), \cr
& = \frac{2}{n-2}d(i,j) + \frac{1}{n-2}\left(\sum_{k\ne i,j} \delta_{\{i,k\}} + \sum_{k\ne i,j} \delta_{\{j,k\}}\right)\cdot d - \frac{2\, \text{Tr}(d)}{(n-1)(n-2)},  \cr
& = \frac{2}{n-2}d(i,j) + \frac{1}{n-2}\left(\sum_{k\ne i,j} d(i,k) + \sum_{k\ne i, j} d(j,k) \right) - \frac{2\, \text{Tr}(d)}{(n-1)(n-2)},  \cr
& = \frac{1}{n-2}\left(\sum_{k\ne i} d(i,k) + \sum_{k\ne j} d(j,k) \right) - \frac{2\, \text{Tr}(d)}{(n-1)(n-2)},  \cr
& = \frac{1}{n-2}\left(s_i + s_j \right) - \frac{2\, \text{Tr}(d)}{(n-1)(n-2)}.
\end{align*}
The calculations above prove most of this theorem: 
\begin{theorem} \label{Main Theorem with Tiger Kung Fu built in}
Assume that $n\ge 5$ and let $w_{sq}$ be the unique solution of the system \eqref{A pair cut metric via matrices} given by $S_{sq}\cdot w_{sq} = d$. Then the $\{i,j\}$-th coordinate of $w_{sq}$ is given by 
\begin{align*}
(w_{sq})_{\{i,j\}} & =  -\frac{1}{2} d(i,j) -\frac{1}{(n-2)(n-4)} \text{Tr}(d) + \frac{1}{2(n-4)}  \left(s_i + s_j \right),	
\end{align*}
with Tr$(d)$ and $s_i$ as in \eqref{Definition of the trace of a metric} and \eqref{Definition of the star-shaped trace of the i-th vertex}.
\end{theorem} 
\begin{proof}
The calculations preceding the theorem proved that 
\begin{align*}
w_{sq} & =  -\frac{1}{2}d - \frac{1}{(n-1)(n-2)(n-4)}\text{Tr}(d) \mathbf{1}_{n\choose 2}  + \frac{n-2}{2(n-4)} P_{Col}(d),
\end{align*}
from which we can compute the $\{i,j\}$-th coordinate of $w_{sq}$ as:  	
\begin{align*}
(w_{sq})_{\{i,j\}} & = -\frac{1}{2} d(i,j) -\frac{1}{(n-1)(n-2)(n-4)} \text{Tr}(d)  + \frac{n-2}{2(n-4)} (P_{\text{Col}}\, d)_{\{i,j\}}, \cr	
& = -\frac{1}{2} d(i,j) -\frac{1}{(n-1)(n-2)(n-4)} \text{Tr}(d) \cr
& \quad  + \frac{n-2}{2(n-4)} \left( \frac{1}{n-2} \left(\sum_{k\ne i} d(i,k) + \sum_{k\ne j} d(j,k) \right) - \frac{2}{(n-1)(n-2)} \text{Tr}(d)\right), \cr	
& = -\frac{1}{2} d(i,j) -\frac{1}{(n-2)(n-4)} \text{Tr}(d) + \frac{1}{2(n-4)}  \left(\sum_{k\ne i} d(i,k) + \sum_{k\ne j} d(j,k) \right), \cr	
& = -\frac{1}{2} d(i,j) -\frac{1}{(n-2)(n-4)} \text{Tr}(d) + \frac{1}{2(n-4)}  \left(\sum_{k\ne i} d(i,k) + \sum_{k\ne j} d(j,k) \right), \cr	
& = -\frac{1}{2} d(i,j) -\frac{1}{(n-2)(n-4)} \text{Tr}(d) + \frac{1}{2(n-4)}  \left(s_i + s_j \right).
\end{align*}
This proves the theorem.   	
\end{proof}
The following is a re-stating of Theorem \ref{Main theorem stated in the introduction} from the introduction, included here for the convenience of the reader. It provides a complete answer to Question \ref{Question about a metric belonging to the pair-cut cone}.
\begin{theorem} \label{Main inequalitis on star-traces and traces of metrics in the pair-cut-cone}
A metric $d$ on $V_n$ with $n\ge 5$, lies in the pair-cut cone $PCUT_n$ if and only if each of the ${n\choose 2}$ inequalities are satisfied:  
\begin{equation} \label{Main inequality for the coordinates of wsq to be nonnegative}
 s_i + s_j\ge (n-4)d(i,j) + \frac{2}{n-2}\text{Tr}(d). 
\end{equation}
These inequalities are enumerated by the pairs $\{i,j\}\subset V_n$, and are best possible (see Example \ref{Example of a single cut semimetric}).
\end{theorem}
\begin{proof} 
From Theorem \ref{Main Theorem with Tiger Kung Fu built in} we see that $(w_{sq})_{\{i,j\}}\ge 0$ if and only if 
\begin{align*}
0 &\le -\frac{1}{2} d(i,j) -\frac{1}{(n-2)(n-4)} \text{Tr}(d) + \frac{1}{2(n-4)}  \left(s_i + s_j \right), \cr	
\frac{1}{2} d(i,j) & \le \frac{1}{2(n-4)}\left(s_i + s_j \right) -\frac{1}{(n-2)(n-4)} \text{Tr}(d), \cr
(n-4)  d(i,j) & \le \left(s_i + s_j \right) -\frac{2\text{Tr}(d) }{(n-2)} .
\end{align*}
This completes the proof. 
\end{proof}

\begin{corollary} \label{Corollary with main ineqaulity for PCUTn membership}
If a metric $d$ on $V_n$ with $n\ge 5$ lies in the pair-cut cone $PCUT_n$, then 
\begin{equation} \label{Secondary inequality on star-traces and traces of metrics in the pair-cut-cone}
 (n-2)s_i \ge  \text{Tr}(d), \quad \text{ for all } i\in V_n. 
 \end{equation}
\end{corollary}
\begin{proof}
This corollary follows from the previous theorem by summing over all indices $j\ne i$:
\begin{align*}
\sum_{j\ne i}s_i +   \sum_{j\ne i} s_j & \ge (n-4)\sum_{j\ne i} d(i,j) + \frac{2(n-1)}{n-2}\text{Tr}(d), \cr
(n-1)s_i +   \sum_{j\ne i} \sum_{k\ne j} d(j,k) & \ge (n-4)s_i + \frac{2(n-1)}{n-2}\text{Tr}(d), \cr
(n-1)s_i +   \sum_{a\ne b} d(a,b) - \sum_{k\ne i} d(i,k)  & \ge (n-4)s_i + \frac{2(n-1)}{n-2}\text{Tr}(d), \cr
(n-1)s_i +    2\text{Tr}(d) - s_i  & \ge (n-4)s_i + \frac{2(n-1)}{n-2}\text{Tr}(d), \cr
(n-2)s_i + 2\text{Tr}(d) & \ge (n-4)s_i + \frac{2(n-1)}{n-2}\text{Tr}(d), \cr
2s_i   & \ge   \frac{2n-2-2n+4}{n-2}\text{Tr}(d), \cr
(n-2)s_i   & \ge \text{Tr}(d).
\end{align*}
\end{proof}
\begin{example} \label{Example of a single cut semimetric}
Consider the $d=\delta_{\{p,q\}}$ on $V_n$, $n\ge 5$, which clearly lies in the pair-cut-cone $PCUT_n$. We use it as an illustration only, to compare the two sides of the inequality \eqref{Main inequality for the coordinates of wsq to be nonnegative}. Explicitly, $d$ is described by 
$$
d(i,j) =\left\{
\begin{array}{cl}
1 & \quad ; \quad |\{i,j\} \cap \{p,q\}| = 1, \cr 
0 & \quad ; \quad |\{i,j\} \cap \{p,q\}| \ne 1. 	
\end{array} 
\right.
$$
It follows that Tr$(d) = 2(n-2)$ and that 
$$
s_i =\left\{
\begin{array}{cl}
	2 & \quad ; \quad i\notin\{ p,q\}, \cr 
	n-2 & \quad ; \quad i\in \{p,q\}. 	
\end{array} 
\right.
$$
The inequality  \eqref{Main inequality for the coordinates of wsq to be nonnegative} then reads 
\begin{align*}
s_i + s_j & \ge  (n-4)d(i,j) +  \frac{2}{n-2}\text{Tr}(d) =  (n-4)d(i,j) +  4.
\end{align*} 
For the various choices of $i,j$ in relation to $p,q$, this latter inequality becomes: 
$$ 
\begin{array}{rl}
2+ 2 - [0 +4  ] = 0 & \quad ; \quad |\{i,j\}\cap \{p,q\}|=0, \cr 
2+ (n-2) - [(n-4) +4  ] = 0 & \quad ; \quad |\{i,j\}\cap \{p,q\}|=1, \cr 
(n-2)+ (n-2) - [0 +4  ] > 0 & \quad ; \quad |\{i,j\}\cap \{p,q\}|=2. 
\end{array} 
$$
In most cases the inequality becomes an equality, and this is what was meant by the comment that ``These inequalities are best possible'' in the statement of Theorem \ref{Main inequalitis on star-traces and traces of metrics in the pair-cut-cone}. The inequalities \eqref{Secondary inequality on star-traces and traces of metrics in the pair-cut-cone} for this example are $2(n-2)\ge 2(n-2)$ and $(n-1)(n-2) \ge 2(n-2)$, showing sharpness in this case too. 
\end{example}
\vskip5mm
\subsection{Examples and Applications} \label{Section with examples and applications of embeddability in PCUT_n}
This section is devoted to applications of Theorem \ref{Main inequalitis on star-traces and traces of metrics in the pair-cut-cone}. We consider several well-known families of graphs, equipped with either the truncated metric $d_0$ or graph metric $d_1$ (cf. Example \ref{Example of SIG metrics from the introduction}) and ask under what conditions, if any, they lie in the pair-cut cone $PCUT_n$. For those metrics not in $PCUT_n$, we comment on what is known about their membership in $CUT_n$, for comparison and contrast. All applications of Theorem \ref{Main inequalitis on star-traces and traces of metrics in the pair-cut-cone} require the computation of the trace Tr$(d)$ of a metric, and the star-traces $s_i$, $i \in V_n$. We shall perform these computations in each example without further explaining their utility. With these quantities computed, we will check if the inequalities \eqref{Main inequality for the coordinates of wsq to be nonnegative} hold. Recall that they read
\begin{equation} \label{Main inequality stated for the second time}
(n-2)(s_i + s_j)\ge (n-4)(n-2)d(i,j) + 2\,{Tr}(d), \quad  \text{ for all } \{i,j\}\subset V_n.  
\end{equation}
Recall also that Theorem \ref{Main inequalitis on star-traces and traces of metrics in the pair-cut-cone} requires $n\ge 5$. 
\vskip3mm
\begin{example}[{\bf Complete Graphs $\mathbf{K_n}$}]
For $n\in \mathbb N$ let $K_n$ denote the complete graph on $n$ vertices enumerated by $V_n = \{1,\dots, n\}$. For this example the truncated and graph metrics agree and we write $d$ to denote either. Each distance $d(i,j)$ equals 1. It follows that 
\begin{itemize}
\item[(i)] Tr$(d) = {n\choose 2}$. 
\item[(ii)] $s_i = n-1$ for all $i\in V_n$. 
\end{itemize} 
Accordingly, inequality \eqref{Main inequality for the coordinates of wsq to be nonnegative} and its simplifications become:
\begin{align*}
2(n-2)(n-1) & \textstyle \ge (n-4)(n-2)+  2 \cdot {n\choose 2},  \cr 
2n^2-6n+4 &\ge (n^2-6n+8) + (n^2-n), \cr 
2n^2-6n+4 &\ge 2n^2-7n+8. 
\end{align*}
\vskip2mm
\noindent {\bf Conclusion for $\mathbf{d}$.} {\em The graph metric $d$ on the complete graph $K_n$ with $n\ge 5$, lies in the pair-cut cone $PCUT_n$.}     
\end{example} 
\begin{example}[{\bf Regular Graphs and Cycles}] \label{Example on regular graphs}
A $k$-regular graph $R_{n,k}$ on $n$ vertices is any graph in which each vertex has valence $k$. The values of $n$ and $k$ do not always fully determine the graph $R_{n,k}$, but we opt to keep our notation simple. Assume that $2\le k$ and note that $k\le n-1$. If $k=n-1$ we obtain the complete graph from the previous example. Thus, for the remainder of this example we assume that $k<n-1$. We only consider the truncated metric $d_0$ as the graph metric $d_1$ depends on the specific example of the $k$-regular graph at hand. For $d_0$ we obtain:
\begin{itemize}
\item[(i)] $s_i = k + 2(n-1-k) = 2n -k-2$ for all $i\in V_n$.  
\item[(ii)] Tr$(d_0) = \frac{1}{2}n(k+2(n-1-k)) = \frac{1}{2}n(2n -k -2)$.
\end{itemize} 
Turning to inequality \eqref{Main inequality for the coordinates of wsq to be nonnegative} and its simplifications, we begin with the case of a pair of vertices $i,j \in V_n$ that don't share an edge so that $d_0(i,j) = 2$:
\begin{align*}
(n-2)(4n-2k-4)  & \ge \textstyle (n-4)(n-2)2+  2\cdot \frac{1}{2}n(2n-k-2), \cr
4n^2-12n -2kn +4k+8 & \ge (2n^2-12n+16) + (2n^2-kn-2n), \cr     
4n^2-12n -2kn +4k+8 & \ge 4n^2-14n -kn +16,   \cr 
2n  +4k & \ge nk   +8.    
\end{align*}
This latter inequality is valid for $k=2$, in which case it becomes an equality. However, if $k\ge 3$ and $n\ge 5$, the inequality fails. The regular graphs $C_n:=R_{n,2}$ are called {\em cycles}, and are uniquely determined by $n$. We still need to consider \eqref{Main inequality for the coordinates of wsq to be nonnegative} for a pair of vertices that share an edge:
\begin{align*}
(n-2)(4n-2k-4)  & \ge \textstyle (n-4)(n-2)+  2\cdot \frac{1}{2}n(2n-k-2), \cr
4n^2-12n -2kn +4k+8 & \ge (n^2-6n+8) + (2n^2-kn-2n), \cr     
4n^2-12n -2kn +4k+8 & \ge 3n^2-8n -kn +8,   \cr 
n^2 +4k & \ge 4n+kn    
\end{align*}
Since we are already restricted to the case of $k=2$ from the previous inequality, this latter inequality becomes $n^2+8 \ge 6n$ which is valid for all $n\ge 4$.
\vskip2mm
\noindent {\bf Conclusion for $\mathbf{d_0}$.} {\em Among the regular graphs $R_{n,k}$ with $2\le k \le n-2$ and with $n\ge 5$, only the truncated metrics for the case of $k=2$ lie in the pair-cut cone. It follows that $d_0$ on cycles $C_n$ with $n\ge 5$, lies in $PCUT_n$.} 
\vskip2mm
We are not aware of studies of the truncated metric $d_0$ on $k$-regular graphs with regards to its membership in the cut cone $CUT_n$. 
\end{example} 
\begin{example}[{\bf Cocktail Party graphs $\mathbf{CP_n}$}]
The cocktail party graphs $CP_n$ are obtained from the complete graph $K_{2n}$ by removing the edges corresponding to a perfect matching. The notation $CP_n$ is incomplete, for simplicity we are suppressing the perfect matching, hoping it will not lead to confusion. All cocktail party graphs are regular graphs of type $R_{2n,2n-2}$ to which the previous example applies, to yield:
\vskip2mm
\noindent  {\bf Conclusion for $\mathbf{d_0}$.} {\em For $n\ge 3$, the truncated metric $d_0$ on a cocktail party graph $CP_n$, does not lie in $PCUT_n$.} 
\vskip2mm
In contrast, the truncated metric $d_0$ lies in $CUT_{2n}$ for all $n\ge 2$ and all perfect matchings. To see this, suppose that $V_{2n}$ is partitioned into $n$ disjoint pairs $P_1, \dots, P_n$ corresponding to some perfect matching on $V_{2n}$. Consider the family $\mathcal T$ of {\em transversal cuts of $V_{2n}$}, given by 
$$\mathcal T = \{C\subset V_{2n}\, |\, |C\cap P_i|=1 \text{ for all } i=1,\dots, n \}.$$ 
Each of the $2^n$ transversal cuts $C$ consists of $n$ elements, one from each pair $P_1, \dots, P_n$. Thus, if $\{x,y\} = P_i$ for some $i$, then $\delta_C (x,y) = 1$ for every $C\in \mathcal T$ and so $\sum_{C\in \mathcal T}\delta_C(x,y) = 2^n$. If $\{x,y\}$ is not a matched pair, then there are $2^{n-2}$ transversal cuts that contain $x$ but not $y$, and the same number of transversal cuts that contain $y$ but not $x$. The cut-metrics $\delta_C$ of these $2^{n-1}$ transversal cuts give $\delta_C(x,y)=1$, while each of the remaining $2^{n-1}$ transversal cuts $D$ give $\delta_D(x,y) = 0$. Therefore $\sum_{C\in \mathcal T}\delta_C(x,y) = 2^{n-1}$. It follows that   
$$ d_0 = \frac{1}{2^{n-1}} \sum _{C\in \mathcal T} \delta_C, $$
and thus $d_0\in CUT_{2n}$. 
\vskip3mm
The analysis of the graph metric $d_1$ on $CP_n$ depends on the specific graph. We consider here the simple example where the perfect matching on $V_{2n}$ is given by the pairs $P_i = \{i,i+n\}$, $i=1,\dots, n$. Then for all $i=1,\dots, n$ and all $j\ne i$, one obtains  
$$d_1(i,j) = \left\{
\begin{array}{cl}   
1 & \quad ; \quad j\ne i+n, \cr 
2 & \quad ; \quad j = i+n,
\end{array}   \right. $$
From this it follows that 
\begin{itemize}
\item[(i)] $s_i = 2+ {2n\choose 2}-2 = {2n\choose 2} = n (2n-1)$ for all $i\in V_n$.  
\item[(ii)] Tr$(d_1) = \frac{1}{2} \sum _{i=1}^{2n}s_i = \frac{1}{2} \cdot (2n) \cdot {2n\choose 2} =   n^2 (2n-1)$.
\end{itemize} 
Thus $s_i+s_j = 4n^2-2n$ for all pairs $\{i,j\}$. Turning to inequality \eqref{Main inequality for the coordinates of wsq to be nonnegative} firstly for a pair of vertices $i$ and $i+n$, we obtain: 
\begin{align*}
(2n-2)(4n^2-2n) & \ge (2n-2)(2n-4)2 + 2n^2(2n-1), \cr
8n^3-12n^2+4n& \ge  (8n^2-24n+16) + (4n^3-2n^2), \cr
8n^3-12n^2+4n& \ge  4n^3+6n^2-24n+16, \cr
4n^3 +28 n& \ge 18n^2 + 16, 
\end{align*}
valid for all $n\ge 2$. Similarly, for a pair of vertices $i,j$ with $j\ne i+n$ one obtain this inequality $4n^3+16n\ge 14n^2+8$, which is also valid for $n\ge 2$. 
\vskip2mm
\noindent {\bf Conclusion for $\mathbf{d_1}$.} {\em For the cocktail party graphs $CP_n$ corresponding to the perfect matching $i\to i+n$ for $i=1,\dots, n$, the graph metric $d_1$ lies in $PCUT_n$ for all $n\ge 3$.}  
\end{example} 
The preceding examples are special in that the star trace $s_i$ is independent of the vertex $i$. Suppose more generally that we have a metric space $(V_n, d)$ with this property, and let $s$ denote the star trace $s_i$ for any choice of $i$. Then the trace of the metric is Tr$(d) = \frac{1}{2}ns$. The inequalities  \eqref{Main inequality for the coordinates of wsq to be nonnegative} then become:
\begin{align*}
(n-2)2s & \ge\textstyle  (n-2)(n-4)d(i,j) + 2 \cdot \frac{1}{2}ns, \cr 
2ns -4s & \ge\textstyle  (n-2)(n-4)d(i,j) + ns, \cr
s(n -4) & \ge\textstyle  (n-2)(n-4)d(i,j), \cr
\frac{s}{n-2} & \ge d(i,j). 
\end{align*}
This proves the following proposition. 
\begin{proposition} \label{Proposition with constant values of star traces}
Let $d$ be a metric on $V_n$ for which $s_i$ is independent of the value of $i\in V_n$, and let $s=s_i$. Then $d\in PCUT_n$ if and only if 
$$d(i,j) \le \frac{s}{n-2},\qquad \text{for all pairs } \{i,j\}\subset V_n.$$ 
\end{proposition}

\vskip3mm
\begin{example}[{\bf Hypercube Graphs $\mathbf{Q_n}$}] \label{Hypercube graphs for the pair-cut cone section}
The hypercube graph $Q_n$ for $n\ge 2$ is the 1-skeleton of the hypercube $[0,1]^n\subset \mathbb R^n$. Two vertices share an edge if and only if their coordinates differ in precisely one entry. Hypercube graphs are examples of $R_{2^n, n}$ regular graphs, and thus the Example \ref{Example on regular graphs} applies. Since Theorem \ref{Main inequalitis on star-traces and traces of metrics in the pair-cut-cone} requires the graph have at least 5 vertices, in the case of $Q_n$ this means we require $n\ge 3$. 
\vskip2mm
\noindent {\bf Conclusion for $\mathbf{d_0}$.} {\em The truncated metric $d_0$ on the hypercube graph $Q_n$ does not lie in the pair-cut cone for any $n\ge 3$.}  
\vskip2mm
It is known that the truncated metric $d_0$ on $Q_n$ lies in the cut cone only for $n= 1, 2$. For $n\ge 3$ this possibility is excluded by a pentagonal hypermetric inequality (for example applied to the vertices $A=(0,0,0,\dots)$, $B=(0,1,1,\dots)$, $C=(0,0,1,\dots)$, $D=(0,1,0,\dots)$ and $E=(1,0,0,\dots)$, with the dots corresponding to all zero coordinates).   
\vskip3mm 
Next consider the graph metric $d_1$ on $Q_n$. For a given vertex, the distance to any other vertex of the hypercube is given by the number of coordinates in which they differ. For a given $k\in \{0,1,\dots, n\}$ there are ${n\choose k}$ vertices that differ from a given one in $k$ coordinates, and thus the star-trace of any vertex $x$ is given by
$$s_x = \sum _{k=0}^n k{n\choose k} = n\,2^{n-1}.$$ 
We see that the star trace is independent of the vertex, allowing us to use Proposition \ref{Proposition with constant values of star traces}, prompting us to examine these inequalities:
$$ d(i,j) \le \frac{s}{2^n-2} = \frac{n\, 2^{n-1}}{2(2^{n-1}-1)}. $$
The distance between $(0,\dots, 0)$ and $(1,\dots, 1)$ is equal to $n$ and represents the maximal distance between any pair of vertices. Thus we need to have satisfied the inequality:
\begin{align*}
n & \le  \frac{n\, 2^{n-1}}{2(2^{n-1}-1)}, \cr
2^n & \le 2^{n-1}+2.
\end{align*}
This inequality only holds for $n=1,2$. 
\vskip2mm
\noindent {\bf Conclusion for $\mathbf{d_1}$} {\em The graph metric $d_1$ on the hypercube graph $Q_n$ does not lie in the pair-cut cone for any $n\ge 3$.}
\vskip2mm
In contrast, the graph metric $d_1$ on $Q_n$ lies in the cut cone $CUT_{2^n}$ for all $n\ge 1$ since $d_1$ coincides with the $\ell_1$-metric on $\mathbb R^n$ restricted to $Q_n$.     
\end{example} 
\vskip3mm
\begin{example}[{\bf Complete Bipartite Graphs $B_{m,n}$}]
The complete bipartite graphs $B_{m,n}$, $m,n\ge 1$, consists of two collection of vertices, $x_1,\dots, x_n$ and $y_1,\dots, y_m$, of which each vertex from one collections is connected by an edge to each vertex from the other collection, but no two vertices from the same collection share an edge. For this example the truncated and graph metrics are equal, and we denote them by $d$. We begin by calculating the trace and star-traces, the latter of which we label as $s(x_i)$ and $s(y_j)$ with obvious meaning. 
\begin{itemize}
\item[(i)] Tr$(d) = mn + n(n-1) + m(m-1)$. 
\item[(ii)] $s(x_i) = m + 2(n-1)$ for all $i=1,\dots n$. 
\item[\phantom{[ii]}] $s(y_j) = n + 2(m-1)$ for all $j=1,\dots m$.
\end{itemize} 
Turning to inequality \eqref{Main inequality for the coordinates of wsq to be nonnegative}, we begin by computing the left-hand side first. 
\begin{align*}
\text{LHS of \eqref{Main inequality stated for the second time}}=(n+m-2)(s(x_{i_1}) + s(x_{i_2})) & = (n+m-2)(2m+4n-4) \cr
&  =4n^2+2m^2+6mn-12n-8m+8, \cr 
(n+m-2)(s(y_{j_1}) + s(y_{j_2})) & = (n+m-2)(4m+2n-4) \cr
&  =2n^2+4m^2+6mn-12m-8n+8, \cr 
(n+m-2)(s(x_{i}) + s(y_{j})) & = (n+m-2)(3m+3n-4) \cr
&  =3n^2+3m^2+6mn-10m-10n+8. 
\end{align*}
Next, to calculate the right-hand side of \eqref{Main inequality stated for the second time}, we distinguish between the cases when the vertices share an edge vs when they don't. Let us call {\em Case 1} the case where they don't share an edge (in which case the vertices are distance 2 apart), and {\em Case 2} the case where they share an edge (making their distance equal to 1). Then the right-hand side of \eqref{Main inequality stated for the second time} becomes:
\begin{align*}
\text{RHS of \eqref{Main inequality stated for the second time}, Case 1} &\textstyle  = 2(m+n-4)(m+n-2)+  2mn + 2n^2-2n+2m^2-2m \cr
& = 4m^2+4n^2+6mn-14m-14n+16, \cr   
\text{RHS of \eqref{Main inequality stated for the second time}, Case 2} &\textstyle  = (m+n-4)(m+n-2)+  2mn+2n^2-2n+2m^2-2m \cr
& = 3m^2+3n^2+4mn-8m-8n+8
\end{align*}
In Case 1, inequality \eqref{Main inequality stated for the second time} becomes one of:
\begin{align*}
4n^2+2m^2+6mn-12n-8m+8 & \ge 4m^2+4n^2+6mn-14m-14n+16, \quad \text{or}, \cr 
2n^2+4m^2+6mn-12m-8n+8 & \ge 4m^2+4n^2+6mn-14m-14n+16
\end{align*} 
which simplify to 
\begin{align*}
n+3m & \ge m^2+4, \cr 
m+3n & \ge n^2+4.
\end{align*} 
Since both inequalities need to hold, by adding them we obtain $0\ge (m-2)^2+(n-2)^2$ leading to $m=n=2$. The inequality for Case 2 is 
$$3n^2+3m^2+6mn-10m-10n+8 \ge 3m^2+3n^2+4mn-8m-8n+8$$
which simplifies to 
\begin{align*}
mn-m-n & \ge 0 \cr 
(m-1)(n-1) & \ge  1.
\end{align*} 
This inequality is satisfied whenever $m,n\ge 2$. In conclusion:
\vskip2mm
\noindent {\bf Conclusion for $\mathbf{d}$.}  For complete bipartite graphs $B_{m,n}$ with $m+n\ge 5$, neither the truncated nor the graph metric lie in $PCUT_{m+n}$.
\vskip2mm
It is known that the graph metrics on $B_{2,2}$ and $B_{1,n}$, $n\ge 1$, lie in the cut cone. Note that $B_{1,n}$ is the Star Graph $S(n)$ which we analyze in the proof of Theorem \ref{Main Result on non-embeddings into the pair-cut-cone}.
\end{example}
\vskip3mm
\begin{example}[{\bf Linear Graphs $L_n$}] 
By a linear graph $L_n$ we mean a closed line segment whose endpoints are a pair of vertices, and which contains an additional $n-2$ interior points as vertices. By necessity, $n\ge 2$. Without loss of generality we shall assume that the segment is $[1,n]\subset \mathbb R$ with the vertices being the integers in this segment. Then the graph metric $d_1$ on $L_n$ is the $\ell_1$-metric on $\mathbb R$, i.e. $d_1(i,j) = |i-j|$ and so clearly $d_1\in CUT_n$ for all $n\ge 2$. From this we can compute the star-trace of $d_1$ for any vertex. If $2i\le n$ then 
\begin{align*}
s_i & = 2\sum_{j=1}^{i-1} j + \sum_{j=2i}^{n}(j-i), \cr 
& = i(i-1) + \sum_{j=0}^{n-2i}(j+i), \cr 
& \textstyle = i(i-1) + (n-2i+1)i + \frac{1}{2}(n-2i+1)(n-2i), \cr
& \textstyle  = i^2-i(n+1)+\frac{1}{2}n(n+1), \qquad 2i\le n. 
\end{align*}
If $2i>n$ then by symmetry $s_i = s_{n-i+1}$ with $2(n-i+1)\le n+1$, leading to 
$$\textstyle s_i = (n-i+1)^2 - (n-i+1)(n+1) + \frac{1}{2}n(n+1) = i^2-i(n+1)+\frac{1}{2}n(n+1), \, 2i>n.$$
The remaining case of $2i=n+1$ is special and leads to $s_i=i^2-i$, which the curious reader can check again agrees with the same formula obtained in the other cases. In conclusion
$$\textstyle  s_i = i^2-i(n+1)+\frac{1}{2}n(n+1), \qquad \text{ for all }i\in V_n.$$
From this we easily obtain the trace of $d_1$:
\begin{align*}
\text{Tr}(d_1)& = \frac{1}{2} \sum_{i=1}^n s_i, \cr
& \textstyle =  \frac{1}{2} \left( \sum_{i=1}^n i^2 - (n+1)\sum_{i=1}^n i + \frac{1}{2}n^2(n+1)\right), \cr 
& \textstyle = \frac{1}{12}n(n+1)(2n+1) -\frac{1}{4}n(n+1)^2+ \frac{1}{4}n^2(n+1), \cr 
& \textstyle  =  \frac{1}{6}(n^3-n).
\end{align*}
For a pair of vertices $1\le i<j\le n$, inequality \eqref{Main inequality for the coordinates of wsq to be nonnegative} becomes:
\begin{align*}
\textstyle (n-2)(i^2+j^2-(i+j)(n+1)+n(n+1)) & \textstyle\ge (n-2)(n-4)(j-i) + \frac{2}{6}(n^3-n).
\end{align*} 
These inequalities have to be satisfied for all values of $1\le i<j\le n$. Choosing $i=1$ and $j=n$ leads to 
\begin{align*}
\textstyle (n-2)(1+n^2-(1+n)(n+1)+n(n+1)) & \textstyle\ge (n-2)(n-4)(n-1) + \frac{2}{6}(n^3-n), \cr  
\textstyle n^3-3n^2+2n &\ge \textstyle \textstyle \frac{4}{3}n^3-7n^2+\frac{41}{3}n-8, \cr 
4n^2 +8 & \ge \textstyle \frac{1}{3}n^3-7n^2+\frac{35}{3}n,  
\end{align*}
which is violated for all $n\ge 9$. 
\vskip2mm
\noindent {\bf Conclusion for $\mathbf{d_1}$.} {\em Let $d_1$ be the graph metric on the linear graph $L_n$. Then for any $n\ge 9$, $d_1\notin PCUT_n$, but $d_1\in CUT_n$ for all $n\ge 2$. } 
\vskip3mm
Turning to the truncated metric $d_0$, and assuming that $n\ge 2$, we obtain 
$$s_i = \left\{ 
\begin{array}{cl}
2n-3 & \quad ; \quad i=1,n, \cr
2n-4 & \quad ; \quad i=2,\dots, n-1 \text{ and } n\ge 3. 
\end{array} 
\right.   
$$
From this one can easily find the trace of the truncated metric: 
$$\text{Tr}(d_0) = \left\{
\begin{array}{cl}
1 & \quad ; \quad n=2, \cr
(n-1)^2 & \quad ; \quad n\ge 3. 
\end{array}
\right.
$$
The smaller star traces occur for the ``interior'' vertices, that is vertices not equal to $1$ and $n$, and so we first consider their corresponding inequality \eqref{Main inequality for the coordinates of wsq to be nonnegative} with $n\ge 4$. The right-hand side is smallest when we consider a pair of interior vertices, while the right-hand side is largest if the vertices don't share an edge. For there to be two such vertices, we require $n\ge 5$. For this case \eqref{Main inequality for the coordinates of wsq to be nonnegative} becomes:
\begin{align*}
(n-2)(4n-8) & \ge 2(n-2)(n-4)+ 2(n-1)^2  \cr
4n^2-16n +16 & \ge 4n^2-16n + 18  \cr
0 & \ge 2,   \cr
\end{align*}
which is clearly violated.  
\vskip2mm
\noindent {\bf Conclusion for $\mathbf{d_0}$.} {\em The truncated metric $d_0$ on the linear graph $L_n$ does not belong to $PCUT_n$ for any $n\ge 5$.}
\vskip2mm
On the other hand, $d_0$ lies in the cut cone $CUT_n$ for all $n\ge 2$, as can be seen from this decomposition:
$$d_0 = \frac{1}{2}\left( \delta_{\{1\}} + \left(\sum_{i=1}^{n-1} \delta_{\{i,i+1\}} \right) + \delta_{\{n\}}\right) \in CUT_n.$$
In particular, $d_ 0 - \frac{1}{2}\delta_{\{1\}}-\frac{1}{2}\delta_{\{n\}}\in PCUT_n$. 
\end{example} 
\section{Simple graphs with no \sig metric in $PCUT_n$} \label{Section with applications of main theorem to simple graphs}
We return to simple graphs $G=(V_n, E)$ and Question \ref{Question about L1-embeddability of simple graphs, as stated in the introduction, into the pair cut cone}, which asks when $G$ possesses a \sigg-metric $d$ that lies in $PCUT_n$. This question requires a two-pronged approach: Firstly, one chooses a \sigg-metric $d$ on $V_n$, and secondly, one checks if $d\in PCUT_n$  with the help of Theorem \ref{Main inequalitis on star-traces and traces of metrics in the pair-cut-cone}. This process is complicated by the fact that a typical graph $G$ admits infinitely many \sigg-metrics, and as we show next, some may and some may not lie in $PCUT_n$. 
\subsection{Embeddability of different \sigg-metrics}
Given a simple graph $G$, it is natural to ask if the property for a \sigg-metric of $G$ to lie in $PCUT_n$ is shared by all \sigg-metrics of $G$. In other words, if one \sigg-metric of $G$ lies in $PCUT_n$, do all of them? The answer to this is unfortunately negative, complicating Question \ref{Question about L1-embeddability of simple graphs, as stated in the introduction, into the pair cut cone}. 
\begin{example}
Consider the graph $G=(V_7, E)$ in Figure \ref{Graph1}. We shall consider the two \sigg-metrics $d_0$ and $d_1$ on $G$, being the truncated metric and graph metric respectively (cf. Example \ref{Example of SIG metrics from the introduction}).  
\begin{figure}[h]
\includegraphics[width=10cm]{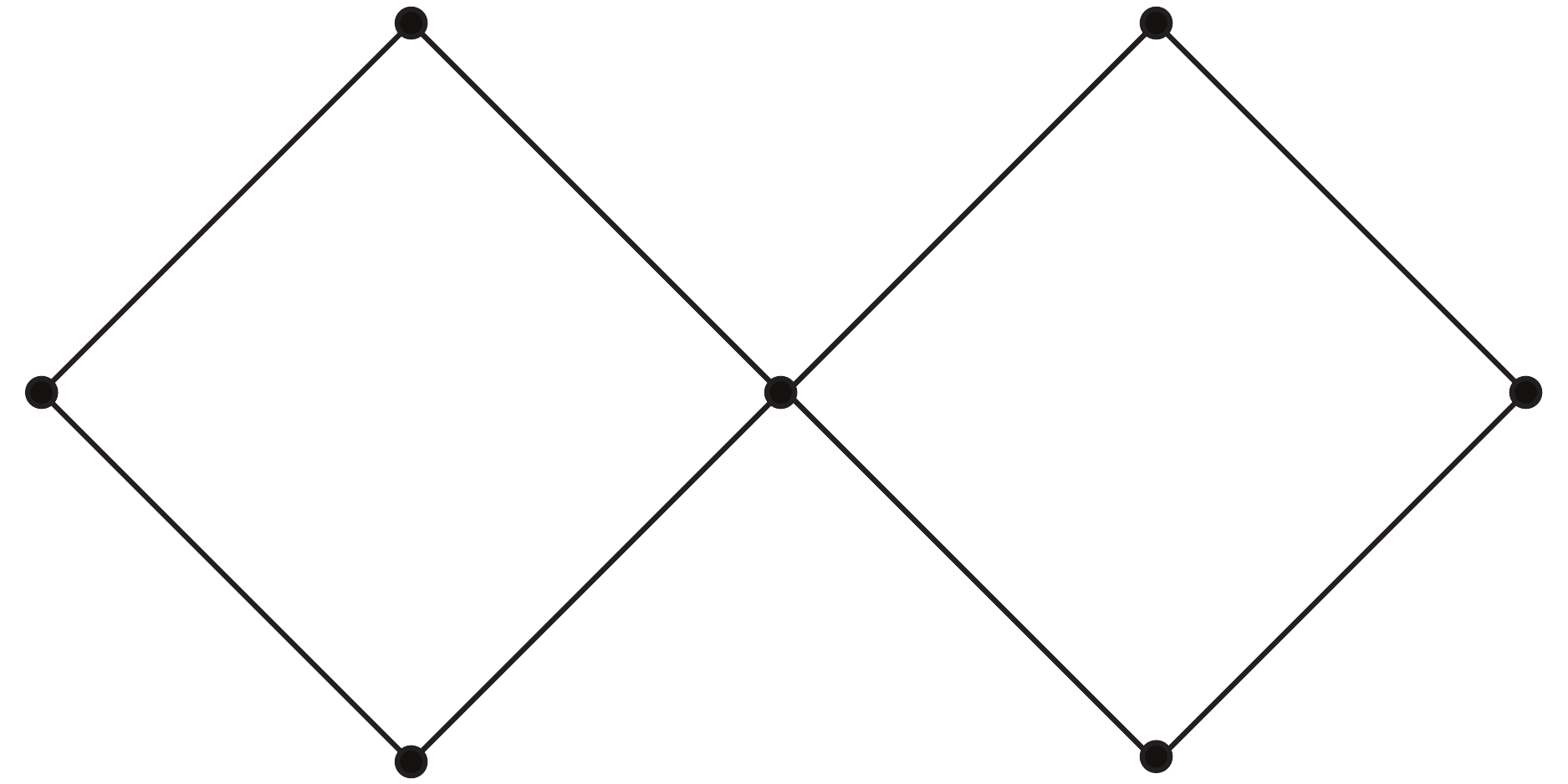}
\put(-145,53){$1$}
\put(-213,-15){$5$}
\put(-78,-15){$4$}
\put(-78,148){$2$}
\put(-213,148){$7$}
\put(-1,67){$3$}
\put(-290,67){$6$}
\caption{A graph with 7 vertices.   } \label{Graph1}
\end{figure}
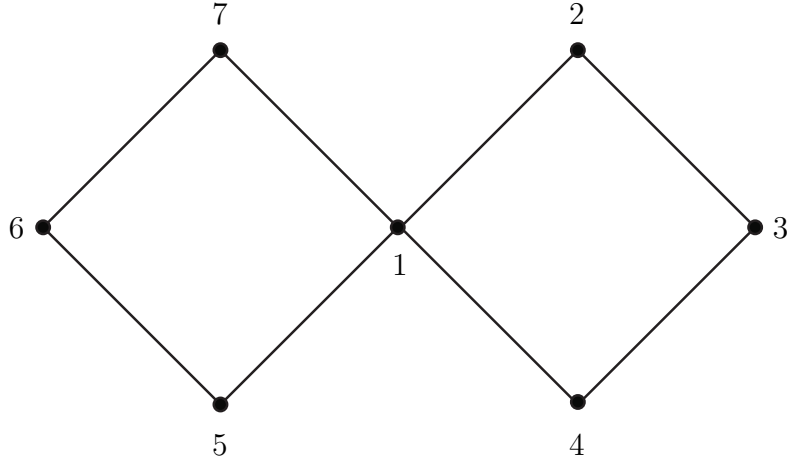
Considering firstly the truncated metric $d_0$, we find: 
$$\text{Tr}(d_0) = 8\cdot 1 + 13\cdot 2 = 34, \quad s_1 = 8, \quad s_i = 10, i=2, \dots, 7. $$
We find that the inequalities \eqref{Main inequality for the coordinates of wsq to be nonnegative} fail for the pairs $\{i,j\} = \{1,3\}, \{1,6\}$:
$$
18 = 8+10  = s_1+s_j \ge (7-4)d(1,j) + \frac{2}{7-2}\text{Tr}(d_0) = \textstyle 19\frac{3}{5}, \quad i=1, j=3,6,
$$
showing that $d_0\notin PCUT_7$. Next we perform the same computation for the graph metric $d_1$, whose values are given by 
$$
\begin{array}{llllll}
d(1,2) =1,\,  & d(1,3) = 2,\,  & d(1,4)=1,\,  & d(1,5) = 1,\,  & d(1,6) = 2,\,  & d(1,7) = 1, \cr 
d(2,3) = 1,\,  & d(2,4)=2,\,  & d(2,5) = 2,\,  & d(2,6) = 3,\,  & d(2,7) = 2, & \cr 
d(3,4)=1,\,  & d(3,5) = 3,\,  & d(3,6) = 4,\,  & d(3,7) = 3, & & \cr 
d(4,5) = 2,\,  & d(4,6) = 3,\,  & d(4,7) = 2, & & & \cr 
d(5,6) = 1,\,  & d(5,7) = 2, & & & & \cr 
d(6,7) = 1, & & & & & \cr 
\end{array}
$$
We see that Tr$(d_1) = 40$ and 
$$s_1 = 8, \quad  s_2 = s_4 = s_5 = s_7= 11, \quad  s_3 = s_6 = 14. $$
The inequalities that need to be satisfied for all pairs $\{i,j\}$ are 
$$ s_i +s_j \ge 3\cdot d(i,j) + 16.$$
It is straightforward to verify that all 21 inequalities hold:
\begin{align*}
 19 = s_1+s_j  & \ge 3+ 16 = 19, \quad j=2,4,5,7, \cr 
 22 = s_1+s_j  & \ge 6+ 16 = 22, \quad j=3, 6, \cr
 22 = s_i+s_j  & \ge 6+ 16 = 22, \quad i\ne j, \, i,j \in\{2,4,5,7\} , \cr
25 = s_i+s_j  & \ge 3+ 16 = 19, \quad \{i,j\} = \{2,3\}, \{3, 4\}, \{6,7\}, \{5,6\},  \cr
25 = s_i+s_j  & \ge 9+ 16 = 25, \quad \{i,j\} = \{2,6\}, \{4, 6\}, \{3, 5\}, \{3,7\},  \cr
28 = s_3+s_6  & \ge 12+ 16 = 28. 
\end{align*}
It follows that $d_1\in PCUT_7$. 
\vskip1mm
In conclusion, the graph $G$ from Figure \ref{Graph1} admits two \sigg-metrics $d_0$ and $d_1$ with $d_0\notin PCUT_7$ and $d_1\in PCUT_7$.  
\end{example}
The following is a re-statement of Theorem \ref{Theorem stated in the introduction with an example of a graph none of whose SIG metrics lie in the pair-cut cone} from the introduction. 
\begin{theorem}\label{Main Result on non-embeddings into the pair-cut-cone}
For any $n\ge 5$ there exist connected simple graphs for which no \sigg-metric lies in the pair-cut cone $PCUT_n$.   
\end{theorem}
\begin{proof} 
The theorem is proved by exhibiting an example of a graph with properties as claimed. Consider the star graph $S(n)$ as in Figure \ref{StarGraph}, with $n\ge 4$. We label the central vertex by $0$ and label the peripheral vertices by elements of $V_n = \{1,\dots, n\}$.  
\begin{figure}[h]
\includegraphics[width=6cm]{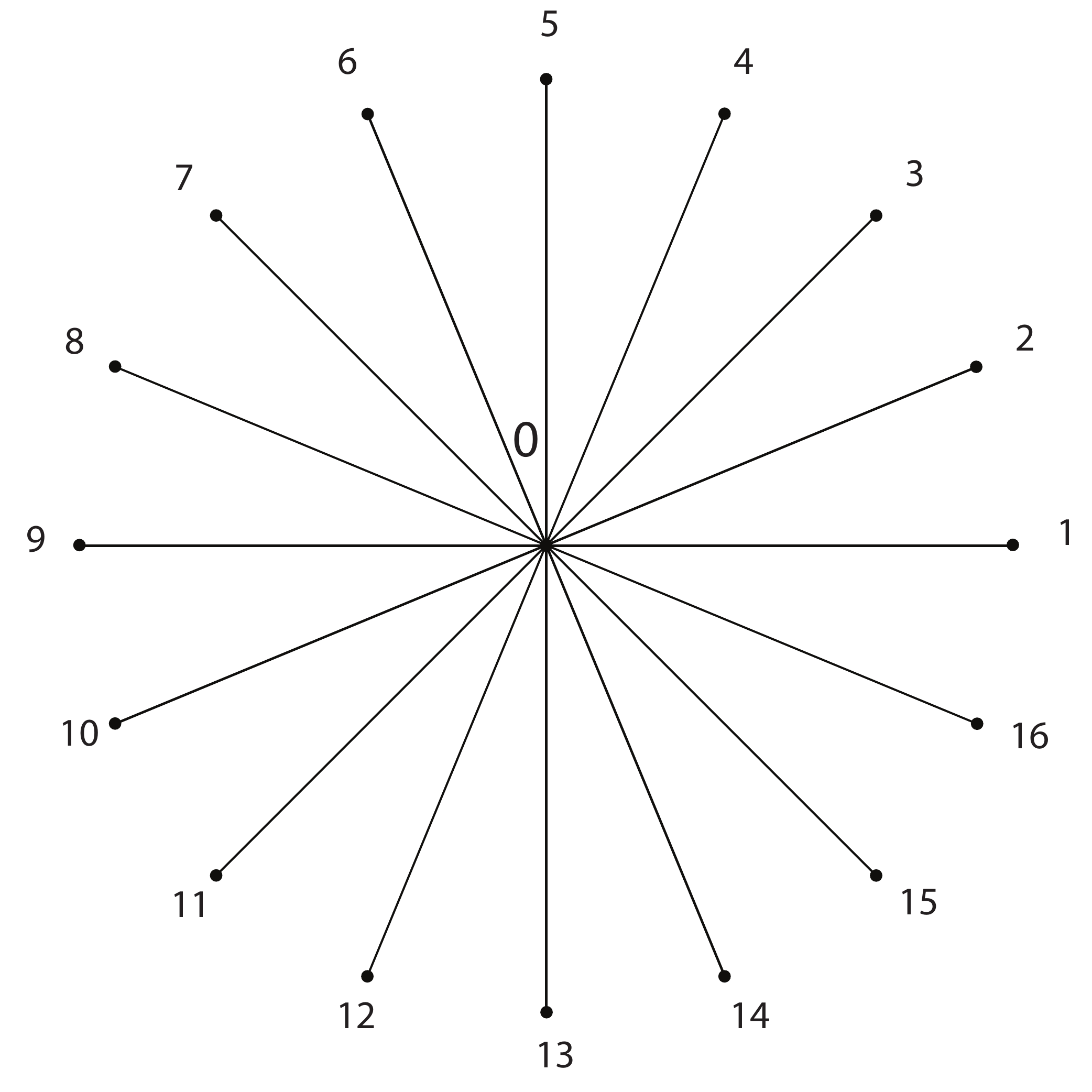}
\caption{The star graph $S(n)$ shown here with $n=16$. The central vertex is labeled $0$ while the peripheral vertices are enumerated by the elements of $V_{16}=\{1,\dots, 16\}$.} \label{StarGraph}
\end{figure} \label{Star Graph from the introduction}
\vskip3mm
\noindent Let $d$ be a \sigg-metric on $W_n= \{0,1,\dots, n\}=\{0\} \cup V_n$ and let $a_i = d(0,i)$. Since each vertex $i \in V_n$ has valence 1, then $r_i = a_i$ for all $i\in V_n$. By the inherent symmetry of $S(n)$, we may assume without loss of generality that 
\begin{equation} \label{assumption on the distances in the star graph}
a_1\le a_2\le \dots \le a_n.
\end{equation}
Thus $r_0 = a_1$. Since $d$ is a \sigg-metric for $G$ and since there are no edges between any two vertices $i,j\in V_n$, it follows from a combination of the triangle inequality and a \sigg-inequality \eqref{The SIG inequalities from the introduction} that  
$$a_i+a_j \le d(i,j) \le  a_i+a_j, \qquad \{i,j\}\subset V_n.$$
Accordingly, $d(i,j) = a_i+a_j$ for $\{i,j\}\subset V_n$. 

We turn to calculating the trace and the star traces of this metric:
\begin{align*}
\text{Tr}(d) & = \sum _{i=1}^n d(0,i) + \sum _{1\le i <j \le n} d(i,j), \cr 
&= \left( \sum _{i=1}^n a_i\right) + \left(\sum _{1\le i <j \le n} a_i+a_j\right), \cr 
&= \left( \sum _{i=1}^n a_i\right) + (n-1)\left(\sum _{i=1}^n a_i    \right), \cr 
&= n \left( \sum _{i=1}^n a_i\right).
\end{align*} 
Let $\alpha = \sum _{i=1}^na_i$ so that Tr$(d) = n\cdot \alpha$. The star traces are also straightforward to compute:
\begin{align*}
s_0& = \sum _{i=1}^n d(0,i) = \sum _{i=1}^n a_i = \alpha, \cr 
& \cr 
s_i & = d(0,i) + \sum _{j\ne 0,i} d(i,j), \cr 
& = a_i + \sum _{j\ne 0,i} (a_i +a_j), \cr 
& = a_i + (n-1)a_i + \sum _{j\ne 0,i} a_j, \cr
& = (n-1)a_i + \sum _{j=1}^n a_j, \cr
& = (n-1)a_i + \alpha. 
\end{align*}
Consider now inequality \eqref{Main inequality for the coordinates of wsq to be nonnegative} for $s_0$ and $s_1$, keeping in mind that $S(n)$ has $n+1$ vertices:
\begin{align*}
s_0+s_1 & \ge (n-3)\, d(0,1) + \frac{2}{n-1}\text{Tr}(d), \cr 
\alpha + (n-1)a_1+\alpha & \ge (n-3)\, a_1 + \frac{2}{n-1}n\alpha, \cr 
\alpha (2n-2) + a_1(2n-2) & \ge 2n\alpha, \cr 
(n-1)a_1 & \ge \alpha, \cr 
(n-1)a_1 & \ge a_1+a_2+\dots +a_n \ge a_1+a_1+\dots +a_1, \cr
(n-1)a_1 & \ge na_1.
\end{align*}
This is a contradiction since $a_1=d(0,1)>0$. It follows that $d\notin PCUT_{n+1}$ and since $d$ was an arbitrary \sigg-metric for $S(n)$, we conclude that $S(n)$ is not a \sigg-graph for any metric $d\in PCUT_{n+1}$, proving the theorem.  
\end{proof}
\section{The Full Cut Matrix} \label{Section on the full cut matrix}
Having established when a metric $d$ on $V_n$ belongs to the pair-cut cone $PCUT_n$, we turn to the question of when $d$ lies in the cut cone $CUT_n$. Here we are not able to provide a full characterization (recall that doing so is an NP-complete problem), but we are able to give a set of conditions that guarantees that $d\in CUT_n$. For these purposes we define the {\em full cut-matrix $S$}, a matrix of dimension ${n \choose 2} \times (2^n-2)$ whose columns are enumerated by the non-trivial cuts of $V_n$, and whose rows are enumerated by pairs in $V_n$. The non-trivial cuts of $V_n$ and the pairs of $V_n$ receive the lexicographic order, and with that convention, the $(k,\ell)$-th entry of $S$ is $\delta_C(\{i,j\})$ if $\{i,j\}$ is the $k$-th pair and $C$ is the $\ell$-th cut. 

Let $m = {n\choose 2}$. An explicit computation shows
\begin{equation} \label{Equation for product of S with Stau}
S\cdot S^\tau = 2^{n-2}(I_m +J_m), 
\end{equation}
with $J_m$, as before, being the $m\times m$ matrix all of whose entries equal 1. Indeed, the ${n\choose 2}\times {n\choose 2}$ matrix $S\cdot S^\tau$ has its row and columns indexed by pairs $\{i,j\}\subset V_n$ with respect to the lexicographic ordering. If $\{i,j\}$ and $\{k,\ell\}$ are the $p$-th and $q$-th pair respectively, then the $(p,q)$-th entry of $S\cdot S^\tau$ is given by 
$$(S\cdot S^\tau)_{p,q} = \sum _{C\subset V_n} \delta_C(i,j)\cdot \delta_C(k,\ell). $$
If $\{i,j\} = \{k,\ell\}$ then
\begin{align*}
(S\cdot S^\tau)_{p,p} & = \sum _{C\subset V_n} \delta_C(i,j),\cr
&\textstyle  = \left(\sum_{i\in C, j\notin C} 1\right) + \left(\sum_{i\notin C, j\in C} 1\right), \cr
& = 2^{n-2} + 2^{n-2},\cr
&  = 2^{n-1}. 
\end{align*}
If $\{i,j\} \ne \{k,\ell\}$ and $|\{i,j\}\cap \{k,\ell\}| = 0$, then we obtain 
\begin{align*}
(S\cdot S^\tau)_{p,q} & = \sum _{C\subset V_n} \delta_C(i,j) \delta_C(k,\ell),\cr 
& = \sum _{i,k\in C, \,  j,\ell\notin C} 1 + \sum _{i,\ell\in C, \,  j,k\notin C} 1 + \sum _{j,k\in C, \,  i,\ell\notin C} 1 + \sum _{j,\ell\in C, \,  i,k\notin C} 1, \cr 
& = 2^{n-4} + 2^{n-4} + 2^{n-4} + 2^{n-4}, \cr 
& = 2^{n-2}.
\end{align*}
Lastly, if $\{i,j\} \ne \{k,\ell\}$ but $|\{i,j\}\cap \{k,\ell\}| = 1$, say $i=k$, then  
\begin{align*}
(S\cdot S^\tau)_{p,q} & = \sum _{C\subset V_n} \delta_C(i,j) \delta_C(i,\ell), \cr 
& = \sum _{i\in C, \,  j,\ell\notin C} 1 + \sum _{j,\ell\in C, \,  i\notin C} 1, \cr 
& 2^{n-3} + 2^{n-3}, \cr 
& = 2^{n-2}.
\end{align*}
This proves equation \eqref{Equation for product of S with Stau}. The matrix $S\cdot S^\tau$ is invertible as one can explicitly check:  
$$ (SS^\tau) ^{-1} = \frac{1}{2^{n-2}} \left(I_m-\frac{1}{m+1}J_m\right),$$
because 
\begin{align*}
(SS^\tau) \cdot \frac{1}{2^{n-2}} & \left(I_m-\frac{1}{m+1}J_m \right)  = (I_m+J_m)\cdot \left(I_m-\frac{1}{m+1}J_m \right), \cr 
& = I_m    -\frac{1}{m+1}J_m  + J_m -\frac{1}{m+1} J_m \cdot J_m, \cr 
& = I_m    -\frac{1}{m+1}J_m  + J_m -\frac{m}{m+1} J_m, \cr 
& = I_m.
\end{align*}
We use this realization we can find a right inverse of $S$, which we shall denote by $S_r$ and which is given by 
$$S_r = S^\tau (SS^\tau)^{-1}\qquad \text{ so that } \qquad SS_r = I_m.$$
Given a metric $d$ on $V_n$, we view it again as the column vector 
$$d = [d(1,2), d(1,3), \dots, d(n-1,n)]^\tau \in \mathbb R^{n\choose 2}.$$ 
Then $d$ belongs to the cut cone $CUT_n$ if and only if the system $Sw=d$ has a solution $w\in \mathbb R^{2^n-2}$ with $w\ge 0$. One solution of $Sw=d$ is given $\mathbf{w} = S_rd$ since 
$$S\mathbf w = SS_r d = I_md = d.$$ 
Before proceeding, we make a definition. To  motivate it, recall that the star trace $s_i = \sum_{j\ne i} d(i,j)$ played an important role in considerations pertaining to the pair-cut cone $PCUT_n$ from Section \ref{Section on the square cut matrix}. The semi-metric $\delta_{\{i\}}$, viewed as a vector in $\mathbb R^m$, has $\{k,\ell\}$-th coordinate equal $1$ if $i\in \{k,\ell\}$, and zero otherwise. It follows then that $s_i = d\cdot \delta_{\{i\}}$, with the right-hand side being the Euclidean inner product on $\mathbb R^m$. Motivated by this observation, we define the {\em $C$-trace $s_C$} for any cut $C\subset V_n$ as 
$$s_C = d\cdot \delta_C   = \sum_{\{i,j\}\subset V_n} \delta_C(\{i,j\})\cdot d(i,j) = \sum_{i\in C,\, j\notin C}d(i,j). $$
Incidentally, 
\begin{align*}
s_i+s_j & = \sum _{k\ne i}d(i,k) + \sum_{\ell\ne j} d(j,\ell) \cr
& = \left(\sum _{a\in\{i,j\}, b\notin\{i,j\}} d(a,b)\right) + 2d(i,j)  \cr 
& =  s_{\{i,j\}} + 2d(i,j), 
\end{align*}
and with this relation, inequality \eqref{Main inequality for the coordinates of wsq to be nonnegative} from Theorem \ref{Main inequalitis on star-traces and traces of metrics in the pair-cut-cone} can be rewritten as 
$$s_{\{i,j\}} \ge (n-6)\, d(i,j) + \frac{2}{n-2}\text{Tr}(d).$$

After this short digression, we are in a position to compute the $C$-coordinate $\mathbf{w}_C$ of $\mathbf w = S_rd$:
\begin{align} \label{Coordinates of the particular solution of the full cut matrix}
\mathbf w_C & = \left( S_r d\right)_C \cr 
& = \left( S^\tau (SS^\tau)^{-1} d\right)_C \cr
& = \frac{1}{2^{n-2}} \left( S^\tau \left(I_m - \frac{1}{m+1}J_m \right)  d\right)_C \cr
& = \frac{1}{2^{n-2}} \left( S^\tau d\right)_C - \frac{1}{m+1}  \left(S^\tau J_m d \right)_C \cr
& = \frac{1}{2^{n-2}} s_C - \frac{\text{Tr}(d)}{m+1}  \left(S^\tau \mathbf{1}_m \right)_C \cr
& = \frac{1}{2^{n-2}} \left( s_C - \frac{\text{Tr}(d)}{m+1} |C|(n-|C|) \right).
\end{align}
In the above we used the fact that the row of $S^\tau$ corresponding to $C$ is $\delta_C^\tau$, so that $(S^\tau d)_C = \delta_C\cdot d = s_C$. We also used the easy-to-verify fact that $J_md = \text{Tr}(d) \cdot \mathbf{1}_m$. Lastly, $(S^\tau \mathbf{1}_m)_C = \delta_C\cdot \mathbf{1}_m$ which equals the sum of the coordinates of $\delta_C$. The non-vanishing coordinates of $\delta_C$, all of which equal to 1, correspond to pairs $\{i,j\}$ with $i\in C$ and $j\notin C$, showing that $(S^\tau \mathbf{1}_m)_C = |C|\cdot(n-|C|)$. This proves the next theorem.
\vskip3mm
\begin{theorem} \label{Main theorem for the full cut matrix and inequalities for a metric to lie in the full cut cone}
If for every cut $C\subset V_n$ we obtain 
\begin{equation} \label{Sufficient condition for embedding in full cut cone}
s_C \ge \frac{|C|(n-|C|)}{{n\choose 2} + 1} \cdot \text{Tr}(d)	
\end{equation} 
then $(V_n,d)\in CUT_n$. 
\end{theorem}
For $C=\{i,j\}$ the inequality \eqref{Sufficient condition for embedding in full cut cone} becomes 
\begin{align*}
s_{\{i,j\}} & \ge \frac{2(n-2)}{\frac{1}{2}(n^2-n)+1} \text{Tr}(d), \cr
s_i+s_j &\ge 2d(i,j) + \frac{2(n-2)}{\frac{1}{2}(n^2-n)+1} \text{Tr}(d), 
\end{align*} 
which is rather different from inequality \eqref{Main inequality for the coordinates of wsq to be nonnegative} in Theorem \ref{Main inequalitis on star-traces and traces of metrics in the pair-cut-cone}. This is what one would expect, for else any metric $d$ that failed one or some of the inequalities \eqref{Main inequality for the coordinates of wsq to be nonnegative} would also fail at least one of the inequalities \eqref{Sufficient condition for embedding in full cut cone}. But of course we expect there to be lots of metrics that lie in the cut cone but that don't lie in the pair-cut cone. 
\begin{example}
Consider again the complete graph $K_n$ on $n$ vertices equipped with the graph metric $d$. Then since $d(i,j)  =1$ for all $i,j\in V_n$, for any cut $C\subset V_n$ we obtain $s_C = |C|(n-|C|)$. The trace of $d$ remains equal to ${n\choose 2}$ making the inequalities \eqref{Sufficient condition for embedding in full cut cone} become 
\begin{align*}
|C|(n-|C|) & \textstyle \ge \frac{|C|(n-|C|)}{{n \choose 2} + 1} \cdot {n\choose 2}, 
\end{align*}
or simply $ {n\choose 2} + 1 \ge {n \choose 2}$, valid for any choice of $C\subset V_n$ and any $n\ge 1$. Indeed, the graph metric is given by $d=\frac{1}{2}\left(\delta_{\{1\}} + \dots + \delta_{\{n\}} \right)$, which clearly lies in $CUT_n$. 
\end{example}

\section{A basis for the kernel of $S$} \label{Section on finding an explicit basis for the full cut-matrix}
Since $S_{sq}$ is the sub-matrix of $S$ corresponding to pair-cuts and since for $n\ge 5$, $S_{sq}$ is a regular matrix by Theorem \ref{Full rank of the cut matrix}, it follows that rk$(S) = {n\choose 2}$. The Rank-Nullity theorem implies that 
\begin{equation} \label{Dimension of the kernel of the full cut matrix}
	\text{rk } \kerr (S) = \textstyle 2^n-2 -{n\choose 2}.
\end{equation}
The general soluion $w$ of $Sw=d$ is obtain as $\mathbf{w} + w_N$ with $\mathbf w$ as in \eqref{Coordinates of the particular solution of the full cut matrix} and with $w_N\in \kerr(S)$. This puts the onus on understanding the kernel of $S$. We are able to determine a basis for $\kerr(S)$, a computation we present next. 

Recall that $\delta_C = \delta_{V_n-C}$ for any cut $C\subset V_n$. It is easy to check that if in the lexicographic ordering of the non-trivial cuts of $V_n$, $C$ is the $k$-th element, then $V_n- C$ is the $(2^n-1-k)$-th element. For later use, we write $\bar C = V_n-C$, an operation which we can think of as an involution on the set of cuts of $V_n$. 
\subsection{The skew-symmetric kernel generators}
Let $\{e_1, \dots, e_{2^n-2}\}$ be the standard basis for $\mathbb R^{2^n-2}$ and for $k=1,\dots, 2^{n-1}-1$, let 
\begin{equation} \label{The skew-symmetric generators of the kernel of the full cut matrix}
\varphi_k = e_k - e_{2^n-1-k}.
\end{equation} 
If $C\subset V_n$ is the $k$-th cut in the lexicographic order of cuts, then we can also write $\varphi_ k  = \delta_C - \delta_{\bar C}$, viewing each cut as one of the standard basis vectors of $\mathbb R^{2^n-2}$ (not to be confused with our earlier convention of viewing $\delta_C$ as a vector in $\mathbb R^{n\choose 2}$). This correspondence reflects the lexicographic ordering of cuts, and thus $\delta_C$ corresponds to $e_k$ if $C$ is the $k$-th cut in the lex ordering. It is trivial to verify that the set $\{\varphi_1, \dots, \varphi_{2^{n-1}-1}\}$ is linearly independent and contained in the kernel of $S$, thus proving:
\begin{proposition} \label{Proposition generating the skew-symmetric kernel elements}
The vectors $\varphi_1, \dots, \varphi_{2^{n-1}-1}$ generate a subspace in $\kerr (S_n)$ of dimension $2^{n-1}-1$. 
\end{proposition}
\begin{example}
Consider $n=4$, then the vectors $\varphi_1, \dots, \varphi_{7}$ are given by 
\begin{align*}
\varphi_1 & = (1,0,0,0,0,0,0,0,0,0,0,0,0,-1), \cr 
\varphi_2 & = (0,1,0,0,0,0,0,0,0,0,0,0,-1,0), \cr
\varphi_3 & = (0,0,1,0,0,0,0,0,0,0,0,-1,0,0), \cr
\varphi_4 & = (0,0,0,1,0,0,0,0,0,0,-1,0,0,0), \cr
\varphi_5 & = (0,0,0,0,1,0,0,0,0,-1,0,0,0,0), \cr
\varphi_6 & = (0,0,0,0,0,1,0,0,-1,0,0,0,0,0), \cr
\varphi_7 & = (0,0,0,0,0,0,1,-1,0,0,0,0,0,0).
\end{align*}
\end{example}
\subsection{The alternating sum kernel generators}
Here we assume that $n\ge 5$. Let $T\subset V_n$ be a set with $3\le |T| \le n$. Viewing non-trivial cuts of $V_n$ as basis vectors in $\mathbb R^{2^n-2}$ we define 
\begin{equation} \label{The alternating sum generators of the kernel of the full cut matrix}
\psi_T = \sum _{C\subset T} (-1)^{|C|}\, \delta_C \in \mathbb R^{2^n-2}. 
\end{equation} 
The sum above is over all cuts $C$ contained in $T$, which are of course also cuts in $V_n$ since $T\subset V_n$. In the sum we include the empty cut $C=\emptyset$ even though $\delta_C =  0$, as this will play an important role in the proof of Proposition \ref{Proof of alternating sums lying in the kernel} below. When translating between cuts and vectors in $\mathbb R^{2^n-2}$, both of the trivial cuts $\delta_{\emptyset}$ and $\delta_{V_n}$ become the zero vector. To get a sense for these kernel generators, it is instructive to look an example. 
\begin{example} \label{Example for n=4 and the symmetric kernel generators}
Consider $n=4$. The lexicographic order of the nontrivial cuts of $V_4$ is given by (along with the correspondence of cuts with vectors from the standard basis for $\mathbb R^{14}$):
$$\begin{array}{c}
\delta_{\{1\}} = e_1, \, \delta_{\{2\}}=e_2, \, \delta_{\{3\}}=e_3, \, \delta_{\{4\}}=e_4, \cr 
\delta_{\{1, 2\}}=e_5, \, \delta_{\{1, 3\}}=e_6, \, \delta_{\{1, 4\}}=e_7, \, \delta_{\{2, 3\}}=e_8,  \, \delta_{\{2, 4\}}=e_9, \, \delta_{\{3, 4\}}=e_{10},\cr  
\delta_{\{1, 2, 3\}}=e_{11}, \, \delta_{\{1, 2, 4\}}=e_{12}, \, \delta_{\{1, 3, 4\}}=e_{13}, \, \delta_{\{2, 3, 4\}}=e_{14}.
\end{array}$$
The vectors $\psi_T$ are created from the choices of $T$ being $T_1=\{1,2,3\}, T_2=\{1,2,4\}, T_3=\{1,3,4\}, T_4=\{2,3,4\}$ and $T_5=\{1,2,3,4\}$. From these we obtain
\begin{align*}
\psi_{T_1} & = (\overbrace{-1,-1,-1,0}^{\text{Singleton cuts}},\overbrace{1, 1, 0, 1, 0, 0}^{\text{Pair cuts}}, \overbrace{-1, 0,0,0}^{\text{Tripe cuts}})\cr
\psi_{T_2} & = (-1,-1,0,-1,1, 0, 1, 0, 1, 0, 0,-1,0,0)\cr
\psi_{T_3} & = (-1,0,-1,-1,0, 1, 1, 0,0,1, 0,0,-1,0)\cr
\psi_{T_4} & = (0,-1,-1,-1,0,0,0,1,1,1,0,0,0,-1)\cr
\psi_{T_5} & = (-1,-1,-1,-1,1,1,1,1,1,1,-1,-1,-1,-1)
\end{align*}
\end{example}
\begin{proposition} \label{Proof of alternating sums lying in the kernel}
The vector $\psi_T$ belongs to the kernel of $S$ for every subset $T\subset V_n$ with $3\le |T|\le n$. 
\end{proposition}
\begin{proof}
The proof is a direct computation. We need to show that the dot product of every row of $S$ with $\psi_T$ equals zero. Let 
$$S_{\{i,j\}} = \text{ row of $S$ corresponding to the pair }\{i,j\}\subset V_n.$$
The columns of $S$ are enumerated by nontrivial cuts $C$ of $V_n$, and the $C$-coordinate of $S_{\{i,j\}}$ equals $\delta_C(\{i,j\})$. This coordinate equals 1 if and only if $C$ contains either $i$ or $j$ but not both. For all other cuts, the coordinate equals 0. There are 3 cases to consider for evaluating the dot product $S_{\{i,j\}}\cdot \psi_T$ according to whether the cardinality of $T\cap \{i,j\}$ equals 0, 1 or 2. 
\vskip3mm
{\bf Case of $T\cap \{i,j\} = \emptyset$.} In this case the non-zero coordinates of $\psi_T$ correspond to cuts $C\subset T\subset V_n$, none of which contain $i$ or $j$. Thus, in the dot product $S_{\{i,j\}} \cdot \psi_T$ we are summing $2^n-2$ terms that are products of a coordinate $S_{\{i,j\}}$ and a coordinate in $\psi_T$ at least one of which is always 0. Accordingly, $S_{\{i,j\}} \cdot \psi_T=0$.
\vskip3mm
{\bf Case of $T\cap \{i,j\} = \{i\}$.} The cuts of $T$ can be divided into two disjoint collections, namely 
\begin{itemize}
	\item[(a)] Cuts $C\subset T-\{i\}$. 
	\item[(b)] Cuts of the form $C\cup \{i\}$ with $C$ as in (a). 
\end{itemize}
None of these cuts contain $j$, and only the cuts in (b) contain $i$. For all other cuts, the corresponding entry of $\psi_T$ equals zero. For the cuts $C$ in part (a), the corresponding coordinates of $S_{\{i,j\}}$ equal zero. It follows that
\begin{align*}
S_{\{i,j\}} \cdot \psi_T & = \sum _{\tiny\begin{array}{c} C\cup\{i\}, \cr  
 C\subset T-\{i\}\end{array}} (-1)^{|C|+1} \cr
& = - \sum _{C\subset T-\{i\}} (-1)^{|C|} \cr 
& = - \sum_{k=0}^{|T|-1}\, \, \,  \sum _{\tiny \begin{array}{c} C\subset T-\{i\}\cr |C|=k \end{array}} (-1)^{k} \cr 
& = - \sum_{k=0}^{|T|-1}\,(-1)^k\, {|T|-1\choose k} \quad \text{(use Binomial Theorem)}  \cr 
& = -(1-1)^{|T|-1} \cr
& = 0. 
\end{align*}
Note that here it was important that the empty cut $C=\emptyset$ be included in the summation over the cuts in $T-\{i\}$ since otherwise the cut $\{i\} = \emptyset \cup \{i\}$ would not have appeared among the cuts in part (b) above.  
\vskip3mm
{\bf Case of $T\cap \{i,j\} = \{i,j\}$.} In this case we partition the cuts of $T$ into 4 disjoint categories:
\begin{itemize}
\item[(a)] Cuts $C\subset T-\{i,j\}$  (since $|T|\ge 3$ by assumption, then $T-\{i,j\}\ne \emptyset$).  
\item[(b)] Cuts $C\cup \{i\}$ with $C$ as in (a). 
\item[(c)] Cuts $C\cup \{j\}$ with $C$ as in (a). 
\item[(d)] Cuts $C\cup \{i,j \}$ with $C$ as in (a).  
\end{itemize} 
The cuts in parts (a) and (d) lead to zero coordinates for $S_{\{i,j\}}$ and thus do not contribute to the scalar product $S_{\{i,j\}} \cdot \psi_T$. For each of the cuts $C\cup\{i\}$ and $C\cup\{j\}$ with $C\subset T-\{i,j\}$, the corresponding coordinate of $S_{\{i,j\}}$ equals 1. Therefore:
\begin{align*}
S_{\{i,j\}} \cdot \psi_T & = \sum  _{\tiny \begin{array}{c} C\cup\{i\} \cr C\subset T-\{i,j\}  \end{array}} (-1)^{|C|+1} + \sum  _{\tiny \begin{array}{c} C\cup\{j\} \cr C\subset T-\{i,j\}  \end{array}} (-1)^{|C|+1} = -2\sum _{C\subset T-\{i,j\}} (-1)^{|C|}.
\end{align*} 
This term vanishes for the same reason as in Case 2 (since $|T-\{i,j\}|\ge 1$), and thus we have shown that each $\psi_T$ lies in the kernel of $S$. 

This last case illuminates the necessity for the condition of $3\le |T|$. If we had $|T|=2$, then $T$ would equal $\{i,j\}$ for some pair of indices $i, j$. The above sum then only contains a single term, corresponding to $C= \emptyset \subset T-\{i,j\} = \emptyset$,  in which case $S_{\{i,j\}} \cdot \psi_T = -2$. We leave it as an exercise to show that $\psi_{\{i\}}$ likewise isn't a kernel element, while $\psi_\emptyset$ is the zero vector. 
\end{proof}
This next result is not needed in the sequel, we provide it for sake of completeness, and without proof.  
\begin{lemma}
If $T_1, T_2\subset V_n$ are subsets with $|T_1|=1$ and $|T_2|=2$, say $T_1=\{i\}$ and $T_2=\{i,j\}$, then the $\{k,\ell\}$-th coordinate of $S\cdot \psi_{T_m}$ for $m=1,2$ is given by:
\begin{align*}
(S\cdot \psi_{T_1})_{\{k,\ell\}} & = \left\{ 
\begin{array}{cl}
1 & \quad ; \quad i\in \{k,\ell\}, \cr 
0 &  \quad ; \quad i\notin \{k,\ell\},
\end{array} 
\right. \cr
& \cr 
(S\cdot \psi_{T_2})_{\{k,\ell\}} & = \left\{ 
\begin{array}{rl}
-2 & \quad ; \quad \{i,j\} = \{k,\ell\}, \cr 
0 &  \quad ; \quad \{i,j\} \ne \{k,\ell\}.
\end{array} 
\right.
\end{align*} 
\end{lemma}
Proposition \ref{Proof of alternating sums lying in the kernel} produces kernel elements $\psi_T$ of $S$ for every subset $T\subset V_n$, $3\le |T|$. There are $2^n-1-{n\choose 2}-n$ such sets. 

\begin{proposition} \label{Proposition showing linear independence of alternating sum vectors}
Let $n\ge 3$. The set 
$$\mathcal V_n = \{\psi_T\,|\, T\subset V_n, \, 1\le |T|\le n\}$$
of $2^n-1$ vectors of dimension $2^n-2$, has the property that if we remove any one vector from it, the remaining $2^n-2$ vectors are linearly independent, and thus a bassis for $\mathbb R^{2^n-2}$. In particular, its subset 
$$\mathcal T_n = \{\psi_T\,|\, T\subset V_n, 3\le|T|\}$$
is linearly independent for all $n\ge 3$. 
\end{proposition}
\begin{proof}
We shall firstly prove that the set $\mathcal W_n = \{\psi_T\,|\, T\subset V_n, 1\le |T|<n\}$ is linearly independent and a basis. 

Start by observing that the set $\{\psi_{\{1\}}, \dots, \psi_{\{n\}}\}$ is linearly independent. This is so because $\psi_{\{i\}}$ has only one non-zero coordinate, namely -1, at the $i$-th coordinate entry.  

Next consider the set $\{\psi_{\{1\}}, \dots, \psi_{\{n\}}, \psi_{\{1,2\}}, \dots, \psi_{\{n-1,n\}}\}$. This set is linearly independent if and only if the modified set is linearly independent in which we have replaced each $\psi_{\{i,j\}}$ by $\psi_{\{i,j\}} - \psi_{\{i\}}-\psi_{\{j\}}$. But the latter now has only one non-vanishing coordinate, namely 1 at the cut $\{i,j\}$, making the linear independence of the modified set obvious. 

One proceeds this way by modifying more generally the element $\psi_{\{i_1,\dots, i_k\}}$ by subtracting vectors associated to cuts of smaller cardinality than $k$, so as to create a vector with only one non-zero coordinate, namely $(-1)^k$ as coordinate entry corresponding to the cut $\{i_1,\dots, i_k\}$. This proved the linear independence of $\mathcal V_n$, showing also that it is a basis for $\mathbb R^{2^n-2}$, since $|\mathcal V_n| = 2^n-2$. 

To complete the proof, we show that 
\begin{equation} \label{Linear combination for psi suv Vn}
\psi_{V_n} = (-1)^{n-1} \sum _{\tiny 
\begin{array}{c} T\subset V_n \cr 1\le |T|<n\end{array}} (-1)^{|T|}\, \psi_T. 
\end{equation}
Given a cut $C\subset V_n$, the $C$-coordinate of the left-hand side equals $(-1)^{|C|}$. Terms on the right-hand side with non-zero $C$-coordinate correspond to those non-empty sets $T\subsetneq V_n$ that contain $C$, in which case the $C$-coordinate $(\psi_T)_C$ of $\psi_T$ also equals $(-1)^{|C|}$. Let $m=|V_n-C| = n-|C|$, then the right-hand side of \eqref{Linear combination for psi suv Vn} becomes   
\begin{align*}
(-1)^{n-1} \sum _{\tiny 
\begin{array}{c} T\subset V_n \cr 1\le |T|<n\end{array}}& (-1)^{|T|}\, (\psi_T)_C  = (-1)^{|C|} \sum _{C\subset T\subsetneq V_n} (-1)^{|T|} \cr
&  = (-1)^{|C|+n-1} \sum _{\tiny \begin{array}{c} T=C\sqcup D\cr D\subsetneq V_n-C\end{array}} (-1)^{|T|}  \cr
&  = (-1)^{|C|+n-1} \sum _{\tiny \begin{array}{c} T=C\sqcup D\cr D\subsetneq V_n-C\end{array}} (-1)^{|C|}\cdot (-1)^{|D|}  \cr
& \textstyle = (-1)^{n-1}\, \sum _{i=0}^{m-1} (-1)^i {m\choose i} \cr
& \textstyle  = (-1)^{n-1}\, \left[\left(\sum _{i=0}^{m} (-1)^i\, {m\choose i}\right) -(-1)^m {m\choose m}\right]  \cr
& \textstyle  = (-1)^{n-1}\, \left[ (1-1)^m -(-1)^m {m\choose m} \right] \cr
& = (-1)^{n+m}  \cr
& = (-1)^{|C|}.
\end{align*}
Since $\mathcal W_n$ is linearly independent and since $\mathcal V_n = \mathcal W_n \cup \{\psi_{v_n}$ and since $\psi_{V_n}$ is linear combination of the vectors from $\mathcal W_n$ with coefficients $\pm$ next to each, then $\mathcal V_n - \psi_T$ is also linear independent, and hence a basis, for any choice of $\emptyset \ne T\subsetneq V_n$, thereby proving the proposition. 
\end{proof}
Propositions \ref{Proposition generating the skew-symmetric kernel elements}, \ref{Proof of alternating sums lying in the kernel} and \ref{Proposition showing linear independence of alternating sum vectors} produce a total of 
$$ \left( 2^{n-1}-1 \right) + \left( 2^n-1-{n\choose 2}-n \right) = \left( 2^n-2 -{n \choose 2}\right)  + \left(2^{n-1}-n \right) $$ 
vectors in the kernel of $S$. Since the dimension of the kernel is $2^n-2 -{n\choose 2}$, we have produced a linearly dependent set of kernel vectors. From these we extract a specific basis. Note that all the vectors from both families all have coordinates equal to $-1$, $0$ or $1$. 
\subsection{A basis for the kernel of $S$} \label{The subsection with the actual basis for the kernel of S}
\begin{proposition} \label{Proposition with the basis for the kernel}
Let as before $\mathcal T_n = \{\psi_T\,|\, T\subset V_n, 3\le|T|\}$ be the set of $2^n-1-{n\choose 2}-n$ elements in the kernel of $S$. Then 
$$\mathcal B_n = \mathcal T_n \cup \{\varphi_i\}_{i=1,\dots, n-1}$$
is a basis for the kernel of $S$.  
\end{proposition}
\begin{proof}
The cardinality of $\mathcal B_n$ agrees with the dimension of $\kerr(S)$, thus we only need to demonstrate the linear independence of $\mathcal B_n$. 

For an index $k\in V_n$ consider the linear functional $P_k:\mathbb R^{2^n-2}\to \mathbb R$ given as 
$$P_k\left( \sum_{\emptyset\subsetneq C \subsetneq V_n} \lambda _C\, \delta_C\right) = \sum \lambda _C \cdot I(k\in C). $$
Here $I(k\in C)$ is the indicator function, returning 1 if $k\in C$ and returning $0$ if $k\notin C$. It is easy to see that $P_k$ is linear as it is the composition of an orthogonal projection (onto the subspace of basis vectors $\delta_C$ with $k\in C$) with the trace functional. 

For any $\psi_T\in \mathcal T_n$ and for any choice of $k\in V_n$, we obtain $P_k(\psi_T) = 0$, as we demonstrate next: If $k\notin T$ then clearly $P_k(\psi_T)=0$. So assume that $k\in T$, then 
\begin{align*}
P_k(\psi_T) & = P_k\left(\sum _{C\subset T} (-1)^{|C|} \delta_C  \right), \cr
& = P_k\left(\sum _{C\subset T-\{k\}} (-1)^{|C|+1} \delta_{C\cup\{k\}}  \right), \cr
& = \sum _{C\subset T-\{k\}} (-1)^{|C|+1},   \cr
& = -\sum _{i=0}^{|T|-1} (-1)^i, \cr 
& = - (1-1)^{|T|-1},\cr
& = 0. 
\end{align*}  
On the other hand, 
$$P_k(\varphi_j) = \left\{ 
\begin{array}{rl}
1 & \quad ; \quad j=k, \cr
-1 & \quad ; \quad j\ne k.
\end{array} 
\right. 
$$
From this we define new linear functionals $L_k:\mathbb R^{2^n-2}\to \mathbb R$ by setting $L_k = P_k - P_{k+1}$, for $k=1,\dots, n-1$. Then it is still true that $L_k(\psi_T)=0$ for every $\psi _T\in \mathcal T_n$. On the other hand,   
\begin{align*}
L_{n-1}(\varphi_j) & = \left\{ 
\begin{array}{rl}
2 & \quad ; \quad j=n-1, \cr
0 & \quad ; \quad j< n-1,.
\end{array} 
\right. \cr
& \cr 
L_{k}(\varphi_j) & = \left\{ 
\begin{array}{rl}
2 & \quad ; \quad j=k, \cr
-2 & \quad ; \quad j=k+1, \cr
0 & \quad ; \quad j\ne k, k+1,
\end{array} 
\right.
\end{align*}
for $k<n-1$. With this understood, suppose that the set $\mathcal B_n$ were linearly dependent and consider a linear dependence relation:
$$\sum _{i=1}^{n-1}\mu_i\, \varphi_i + \sum _{\psi_T\in \mathcal T_n}\lambda _T\,  \psi_T = 0. $$
Apply firstly the linear functional $P_{n-1}$ to both sides to obtain $\mu_{n-1} = 0$. Now consecutively applying $P_{n-2}, \dots, P_1$ to the linear dependence relation further yields $\mu_{n-2} =0, \dots, \mu_1=0$. Since we already showed that $\mathcal T_n$ is linearly independent (cf. Proposition \ref{Proposition showing linear independence of alternating sum vectors}), it follows that $\lambda _T=0$ for all $T$, proving the proposition. 
\end{proof}
\begin{remark}
The preceding proposition can be generalized to show that in fact each set 
$$ \mathcal T_n \cup \{\varphi_{i_1}, \dots, \varphi_{i_{n-1}}\},$$
for any choice of indices $1\le i_1<\dots<i_{n-1}\le 2^{n-1}-1$ is linearly independent and thus also a basis for the kernel of $S$. The proof is more convoluted but follows similar ideas to the proof already provided. 
\end{remark}
\subsection{Concluding remarks} 
The results of Sections \ref{Section on the full cut matrix} and \ref{The subsection with the actual basis for the kernel of S} provide a pathway to determining if a metric $d$ on $V_n$ lies in the cut-cone $CUT_n$: One finds an initial solution $\mathbf w$ of the linear system $Sw=d$ with 
$$\mathbf w_C =  \frac{1}{2^{n-2}} \left( \left( \sum_{i\in C, j\notin C} d(i,j)\right) - \frac{\text{Tr}(d)}{{n \choose 2}+1} |C|(n-|C|) \right),$$
where $C\subset V_n$ is a non-trivial cut. If $\mathbf w\ge 0$, $d$ lies in the cut cone. If not, one examines the linear combinations 
$$\mathbf{w} + \sum_{i=1}^{n-1}\mu_i \,\varphi_i + \sum _{\psi_T\in \mathcal T_n}\lambda _T\, \psi_T,$$
with $\mu_i, \lambda _T\in \mathbb R$, and with the basis elements $\mu_i, \psi_T$ of $S$ as in Proposition \ref{Proposition with the basis for the kernel}. Then $d\in CUT_n$ if and only if among these vectors there is one with all coordinates non-negative. This is computationally challenging, especially as $n$ grows, but not impossible for concrete examples. We leave an exploration of this approach for future work.  

\bibliographystyle{plain}
\bibliography{bibliography}
\end{document}